\documentclass[reqno]{amsart}
\usepackage{amssymb}
\usepackage{graphicx}

\usepackage[usenames, dvipsnames]{color}
\usepackage{verbatim}
\usepackage{mathrsfs}
\usepackage{bm}
\usepackage{cite}

\numberwithin{equation}{section}

\newtheorem{theorem}{Theorem}[section]

\newtheorem{lemma}[theorem]{Lemma}
\newtheorem{prop}[theorem]{Proposition}

\theoremstyle{definition}
\newtheorem{remark}[theorem]{Remark}

\theoremstyle{definition}

\theoremstyle{definition}

\makeatletter
\def\dashint{\operatorname%
{\,\,\text{\bf-}\kern-.98em\DOTSI\intop\ilimits@\!\!}}
\makeatother

\def\\det{\text{det}}

\def\.5{\frac{1}{2}}

\newcommand{\RN}[1]{%
  \textup{\uppercase\expandafter{\romannumeral#1}}%
}

\renewcommand{\epsilon}{\varepsilon}

\newcounter{marnote}

\begin{document}


\title[Blowup Analysis for the Perfect Conductivity Problem ]{Blowup Analysis for the Perfect Conductivity Problem with convex but not strictly convex inclusions}

\author[H.J. Ju]{Hongjie Ju}
\address[H.J. Ju]{School of  Sciences, Beijing University of Posts  and Telecommunications,
Beijing 100876, China}
\email{hjju@bupt.edu.cn}
\thanks{H.J. Ju   was partially supported by NSFC (11301034) (11471050).}

\author[H.G. Li]{Haigang Li}
\address[H.G. Li]{School of Mathematical Sciences, Beijing Normal University, Laboratory of Mathematics and Complex Systems, Ministry of Education, Beijing 100875, China. }
\email{hgli@bnu.edu.cn}
\thanks{H.G. Li was partially supported by  NSFC (11571042) (11371060) (11631002), Fok Ying Tung Education Foundation (151003). }

\author[L.J. Xu]{Longjuan Xu}
\address[L.J. Xu]{School of Mathematical Sciences, Beijing Normal University, Laboratory of Mathematics and Complex Systems, Ministry of Education, Beijing 100875, China.}
\email{ljxu@mail.bnu.edu.cn. Corresponding author.}

\maketitle

\begin{abstract}
In the perfect conductivity problem, it is interesting to study whether the electric field can become arbitrarily large or not, in a narrow region between two adjacent perfectly conducting inclusions. In this paper, we show that the relative convexity of two adjacent inclusions plays a key role in the blowup analysis of the electric field and find some new phenomena. By energy method, we prove the boundedness of the gradient of the solution if two adjacent inclusions fail to be locally relatively strictly convex, namely, if the top and bottom boundaries of the narrow region are partially ``flat". The boundary estimates when an inclusion with partially ``flat" boundary is close to the ``flat" matrix boundary and estimates for the general elliptic equation of divergence form are also established in all dimensions.
\end{abstract}

\section{Introduction and main results}

In this paper, we mainly investigate the significant role of the relative convexity of two adjacent inclusions in the blowup analysis of the electric field. The problem arises from the study of high-contrast fiber-reinforced composite materials. It is well known that high concentration phenomenon of extreme electric field or mechanical loads in the extreme loads will be amplified by the composite microstructure, for example, the narrow region between two adjacent inclusions or the thin gap between the inclusions and the matrix boundary. However, when the narrow region has certain partially ``flat" top and bottom boundaries (see Figure \ref{fig1}), we prove that the electric field, represented by the gradient of the solutions $|\nabla u|$, is bounded by some positive constant, which is independent of the distance between the inclusions, rather than blow up as one might suppose. Since the antiplane shear model of composite material is consistent with the two-dimensional conductivity model, our results here have also a valuable meaning for the damage analysis of composite materials.

For strictly convex inclusions, especially for circular inclusions, there have been many important works on the gradient estimates. For two adjacent disks with $\varepsilon$ apart, Keller \cite{k1} was the first to use analysis to estimate the effective properties of particle reinforced composites. In \cite{basl}, Babu\v{s}ka, Andersson, Smith, and Levin numerically analyzed the initiation and growth of damage in composite materials, in which the inclusions are frequently spaced very closely and even touching. Bonnetier and Vogelius \cite{bv} and Li and Vogelius \cite{lv} proved the uniform boundedness of $|\nabla{u}|$ regardless of $\varepsilon$ provided that the conductivities stay away from $0$ and $\infty$.
Li and Nirenberg \cite{ln} extended the results in \cite{lv} to general divergence form second order elliptic systems including systems of elasticity. On the other hand, in order to investigate the high-contrast conductivity problem, Ammari, Kang, and Lim \cite{akl} were the first to study the case of the close-to-touching regime of disks whose conductivities degenerate to $\infty$ or $0$, a lower bound on $|\nabla u|$ was constructed there showing blowup of order $\varepsilon^{-1/2}$ in dimension two. Subsequently, it has been proved by many mathematicians that for the two close-to-touching inclusions case the generic blowup rate of $|\nabla{u}|$ is $\varepsilon^{-1/2}$ in two dimensions, $|\varepsilon\log\varepsilon|^{-1}$ in three dimensions, and $\varepsilon^{-1}$ in dimensions greater than four. See Ammari, Kang, Lee, Lee and Lim \cite{aklll}, Bao, Li and Yin \cite{bly1,bly2}, as well as Lim and Yun \cite{ly,ly2}, Yun \cite{y1,y2,y3}, Lim and Yu \cite{lyu}. The corresponding boundary estimates when one inclusion close to the boundary was established in \cite{aklll} mentioned before for disk inclusion case and in \cite{LX} by Li and Xu for the general convex inclusion case in all dimensions. Further, more detailed, characterizations of the singular behavior of gradient of $u$ have been obtained by Ammari, Ciraolo, Kang, Lee and Yun \cite{ackly}, Ammari, Kang, Lee, Lim and Zribi \cite{akllz}, Bonnetier and Triki \cite{bt1,bt2}, Gorb and Novikov \cite{gn} and Kang, Lim and Yun \cite{kly,kly2}. We draw the attention of readers that recently, Bao, Li and Li \cite{bll, bll2} obtained the pointwise upper bound of the gradient of solution to the Lam\'{e} system with partial infinite coefficients, where an iteration technique for energy estimate overcomes the difficulty in the study of elliptic systems caused by the lack of the maximum principle, which is an essential tool to deal with the scalar case. The boundary estimates was studied by Bao, Ju and Li \cite{bjl}. Kang and Yu \cite{ky} by using the layer potential techniques and the singular functions obtained a lower bound of the gradient of solution in dimension two and shows that the blowup rate in \cite{bll} is optimal. For more related work on elliptic equations and systems from composites, see \cite{adkl,bc,dongli,dongzhang,kly0,llby,m} and the references therein.

As we have mentioned before, in all the known work above, the strict convexity of the inclusions (or at least the strictly relative convexity of two adjacent inclusions) are assumed. Interestingly, when the inclusions are only convex but not strictly convex (see Figure \ref{fig1}), we find in this paper that blowup will not occur any more for the perfect conductivity problem. We prove that $|\nabla u|$ is uniformly bounded with respect to $\varepsilon$ whenever the area of the flat boundaries is bigger than zero. To the best of our knowledge, this is a new phenomenon in the blowup analysis of perfect conductivity problem. The corresponding result for linear elasticity case will be presented in a forthcoming paper.

\begin{figure}[t]
\begin{minipage}[c]{0.9\linewidth}
\centering
\includegraphics[width=1.5in]{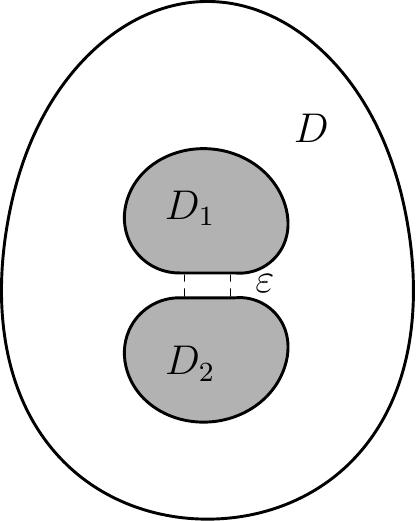}
\caption{\small Two adjacent inclusions with partially ``flat" boundaries.}
\label{fig1}
\end{minipage}
\end{figure}

To describe the problem and results, we first fix our domain and notations. Let $D$ be a bounded open set in $\mathbb R^{n}$ $(n\geq2)$. $D_{1}^{0}$ and $D_{2}^{0}$ be a pair of (touching) subdomains of $D$ with $C^{2,\alpha}$ $(0<\alpha<1)$ boundaries and far away from $\partial D$. They are convex but not strictly convex, having a part of common boundary $\Sigma'$, such that
$$D_{1}^{0}\subset\{(x',x_{n})\in\mathbb R^{n}| x_{n}>0\},\quad D_{2}^{0}\subset\{(x',x_{n})\in\mathbb R^{n}| x_{n}<0\},$$
and
$$\partial D_{1}^{0}\cap\partial D_{2}^{0}=\Sigma'\subset\mathbb R^{n-1}.$$
We assume that $\Sigma'$ is a bounded convex domain in $\mathbb R^{n-1}$, which  can contain an $(n-1)$-dimensional ball. We set the center of the mass of $\Sigma'$ to be the origin. We also assume that  the $C^{2,\alpha}$ norms of $\partial{D}_{1}^{0}$, $\partial{D}_{2}^{0}$ and  $\partial{D}$ are bounded by some positive constant. By translating $D_{1}^{0}$ by a positive number $\varepsilon$ along $x_{n}$-axis, while $D_{2}^{0}$ is fixed, we obtain $D_{1}^{\varepsilon}$, that is,
$$D_{1}^{\varepsilon}:=D_{1}^{0}+(0',\varepsilon).$$
When there is no possibility of confusion, we drop superscripts and denote
$$D_{1}:=D_{1}^{\varepsilon},\quad D_{2}:=D_{2}^{0},\quad \mbox{and}~ \Sigma:=\Sigma'\times(0,\varepsilon).$$
See Figure \ref{fig1}. We use superscripts prime to denote the $(n-1)$-dimensional domains and variables, such as $\Sigma'$ and $x'$.

Further, we may assume that there exists a constant $R_{1}$, independent of $\varepsilon$, such that $\Sigma'\subset B'_{R_{1}}$ and the top and bottom boundaries of the narrow region between $\partial{D}_{1}$ and $\partial{D}_{2}$ can be represented as follows. The corresponding partial boundaries of $\partial D_{1}$ and $\partial D_{2}$ are, respectively,
\begin{align}\label{h1h2'}x_{n}=\varepsilon+h_{1}(x')\quad\mbox{and}\quad x_{n}=h_{2}(x'),\quad\mbox{for}~x'\in B'_{R_{1}},\end{align}
with
\begin{equation}\label{h1h2}
h_{1}(x')=h_{2}(x')\equiv0,\quad\mbox{for}~~x'\in\Sigma'.
\end{equation}
Moreover, in view of the assumptions of $\partial D_1$ and $\partial D_2$,  $h_{1}$ and $h_{2}$ satisfy
\begin{equation}\label{h1-h2}
\varepsilon+h_{1}(x')>h_{2}(x'),\quad\mbox{for}~~x'\in B'_{R_{1}}\setminus\overline{\Sigma'},
\end{equation}
\begin{equation}\label{h1---h2}
\nabla_{x'}h_{1}(x')=\nabla_{x'}h_{2}(x')=0,\quad\mbox{for}~~x'\in\partial\Sigma',
\end{equation}
\begin{equation}\label{h1----h2}
\nabla^{2}_{x'}(h_{1}-h_{2})(x')\geq\kappa_{0}I_{n-1},\quad\mbox{for}~~x'\in B'_{R_{1}}\setminus\overline{\Sigma'},
\end{equation}
and
\begin{equation}\label{h1h3}
\|h_{1}\|_{C^{2,\alpha}(B'_{R_{1}})}+\|h_{2}\|_{C^{2,\alpha}(B'_{R_{1}})}\leq{\kappa_1},
\end{equation}
where $\kappa_{0},\ \kappa_1$ are positive constant, $I_{n-1}$ is the $(n-1)\times(n-1)$ identity matrix. Denote $\Omega:=D\setminus\overline{D_1\cup D_2}$.

Suppose that the conductivities of the inclusions $D_{1}$ and $D_{2}$ degenerate to $\infty$; in other words, the inclusions are perfect conductors. Consider the following perfect conductivity problem
\begin{equation}\label{equinfty1}
\begin{cases}
\Delta{u}=0,& \mbox{in}~\Omega,\\
u=C_{i},&\mbox{on}~\overline{D}_{i},~i=1,2,\\
\int_{\partial{D}_{i}}\frac{\partial{u}}{\partial\nu}\big|_{+}=0,& i=1,2,\\
u=\varphi,&\mbox{on}~\partial{D},
\end{cases}
\end{equation}
where $\varphi\in{C}^{2}(\partial D)$, $C_{i}$ are some constants to be determined later, and
$$\frac{\partial{u}}{\partial\nu}\bigg|_{+}:=\lim_{\tau\rightarrow0}\frac{u(x+\nu\tau)-u(x)}{\tau}.$$
Here and throughout this paper $\nu$ is the outward unit normal to
the domain and the subscript $\pm$ indicates the limit from outside
and inside the domain, respectively.

The existence, uniqueness and regularity of solutions to problem
\eqref{equinfty1} can be referred to the Appendix in \cite{bly1}, with a minor modification. Throughout the paper, unless otherwise stated, we use $C$ to denote some positive constant, whose values may vary from line to line, depending only on $n, \kappa_0, \kappa_1, R_{1}$, and an upper bound of the $C^{2,\alpha}$ norms of $\partial D_{1}, \partial D_{2}$ and $\partial D$, but not on $\varepsilon$. We call a constant having such dependence a {\it universal constant}.

Denote
\begin{equation*}
\Omega_r:=\left\{(x',x_{n})\in \mathbb{R}^{n}~\big|~h_{2}(x')<x_{n}<\varepsilon+h_{1}(x'),~|x'|<r\right\}, \quad\mathrm{for}\ 0<r\leq\,R_{1},
\end{equation*}
and
\begin{equation}\label{def_rhon}
\rho_{n}(\varepsilon)=
\begin{cases}
\sqrt{\varepsilon},& n=2,\\
\varepsilon|\mathrm{ln}\varepsilon|, & n=3,\\
\varepsilon, & n\geq4.
\end{cases}
\end{equation}
Under the assumptions as above, we have the following gradient estimates in all dimensions.

\begin{theorem}\label{thm1}
Suppose that $D_{1}, D_{2}\subset{D}\subset\mathbb{R}^{n}$ ($n\geq2$) be defined as above and  (\ref{h1h2'})--(\ref{h1h3}) hold. Let
$u\in{H}^{1}(D)\cap{C}^{1}(\overline{\Omega})$
be the solution to \eqref{equinfty1}. Then  we have
\begin{align}\label{upperbound1'}
|\nabla{u}(x',x_{n})|\leq \frac{C\varepsilon}{|\Sigma'|+\rho_{n}(\varepsilon)}\frac{1}{\varepsilon+dist^{2}(x',\Sigma')}||\varphi||_{C^{2}(\partial D)},\qquad\mbox{for}~x\in\Omega_{R_{1}}.
\end{align}
and
\begin{align}\label{upperbound10}
||\nabla{u}||_{L^{\infty}(\Omega\setminus\Omega_{R_{1}})}\leq C||\varphi||_{C^{2}(\partial D)}.
\end{align}
\end{theorem}

\begin{remark}\label{rmk1}
From \eqref{upperbound1'}, we can see that if $|\Sigma'|>0$, where $|\Sigma'|$ denotes the area of $\Sigma'$, then for sufficiently small $\varepsilon>0$ (such that $\rho_{n}(\varepsilon)<|\Sigma'|$), $|\nabla u|$ is also bounded in $\Omega_{R_{1}}$. This implies that there is no blowup occurring whenever $|\Sigma'|>0$. When $|\Sigma'|=0$ (that is, $\Sigma'=\{0'\}$), \eqref{upperbound1'} becomes
$$|\nabla{u}(x',x_{n})|\leq \frac{C\varepsilon}{\rho_{n}(\varepsilon)}\frac{1}{\varepsilon+|x'|^{2}}||\varphi||_{C^{2}(\partial D)},\qquad\mbox{for}~x\in\Omega_{R_{1}}.
$$
This pointwise upper bound estimates shows that
$$|\nabla{u}(0',x_{n})|\leq \frac{C}{\rho_{n}(\varepsilon)}||\varphi||_{C^{2}(\partial D)},\quad 0<x_{n}<\varepsilon.
$$
From the proof of Theorem \ref{thm1}, we also can obtain the lower bound of $|\nabla{u}(0',x_{n})|$,
$$|\nabla{u}(0',x_{n})|\geq \frac{1}{C\rho_{n}(\varepsilon)},
$$
whenever some linear functional of $\varphi$ is not equal to zero, see Remark \ref{rem2.5} after the proof of Theorem \ref{thm1}. This shows that the blowup rate is $\frac{1}{\rho_{n}(\varepsilon)}$, which is consistent with the known results, such as \cite{bly1,akl,aklll,ly}.
\end{remark}

\begin{remark}
We draw attention of readers that although the strictly convexity of $D_{1}$ and $D_{2}$ is not necessary in some known work, such as \cite{bly1,kly}, where
$$\frac{1}{C}|x'|^{m}\leq (h_{1}-h_{2})(x')\leq C|x'|^{m}, \quad\mbox{for}~|x'|<R_{1}$$
is assumed, the assertion on boundedness of $|\nabla u|$ when $|\Sigma'|>0$ can not be seen by simply sending $m\rightarrow\infty$ by the methods used there. The fact can also be observed in Subsection \ref{general case}, where more general $D_{1}$ and $D_{2}$ are considered.
\end{remark}

Next, we consider another interesting case when one convex but not strictly convex inclusion $D_{1}$ is very close to the boundary $\partial D$ with partial flatness. See Figure \ref{fig2}. Denote $\widetilde{\Omega}:=D\setminus\overline{D}_1$. The smoothness assumptions are the same as before, just replacing the assumption of $\partial D_{2}$ to $\partial D$, correspondingly, $h_{2}(x')$ to $h(x')$.

Similarly, we consider the following perfect conductivity problem
\begin{equation}\label{equinfty2}
\begin{cases}
\Delta{u}=0,& \mbox{in}~\widetilde{\Omega},\\
u=C_{1},&\mbox{on}~\overline{D}_{1},\\
\int_{\partial{D}_{1}}\frac{\partial{u}}{\partial\nu}\big|_{+}=0,&\\
u=\varphi,&\mbox{on}~\partial{D}.
\end{cases}
\end{equation}
Then we have

\begin{figure}[t]
\begin{minipage}[c]{0.9\linewidth}
\centering
\includegraphics[width=1.8in]{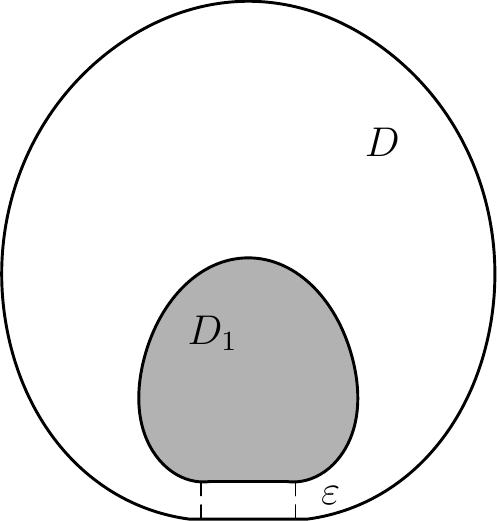}
\caption{\small One ``flat'' inclusion close to the matrix boundary.}
\label{fig2}
\end{minipage}
\end{figure}

\begin{theorem}\label{thm2}
Let $D_{1}\subset{D}\subset\mathbb{R}^{n}$ ($n\geq2$) be defined as above. Let
$u\in{H}^{1}(D)\cap{C}^{1}(\widetilde{\Omega})$
be the solution to \eqref{equinfty2}, then we have
\begin{align}\label{upperbound2'}
|\nabla{u}(x)|
&\leq\frac{C\varepsilon}{|\Sigma'|+\rho_n(\varepsilon)}\frac{|\widetilde{Q}[\varphi]|}{\varepsilon+dist^2(x',\Sigma')} +\frac{|\varphi(x', h(x'))-\varphi(0)|}{\varepsilon+dist^2(x',\Sigma')}+C||\varphi||_{C^{2}(\partial D)},\quad\,x\in\Omega_{R_{1}},
\end{align}
and
\begin{align}\label{upperbound2 2'}
|\nabla{u}(x)|
\leq C|\widetilde{Q}[\varphi]|+C||\varphi||_{C^{2}(\partial D)},\quad\,x\in\widetilde{\Omega}\setminus\Omega_{R_{1}},
\end{align}
where 
\begin{equation*}
\widetilde{Q}[\varphi]=\int_{\partial{D}_{1}}\frac{\partial{v}_{0}}{\partial\nu}
\end{equation*}
is a bounded functional of $\varphi$, and $v_0\in{C}^{2}(\widetilde{\Omega})$ is the solution of
\begin{equation}\label{equ_v0}
\begin{cases}
\Delta{v}_{0}=0,&\mathrm{in}~\widetilde{\Omega},\\
v_{0}=0,&\mathrm{on}~\partial{D}_{1},\\
v_{0}=\varphi(x)-\varphi(0),&\mathrm{on}~\partial{D}.
\end{cases}
\end{equation}
\end{theorem}

\begin{remark}
(1)  From (\ref{upperbound2'}), we can see that if $|\Sigma'|>0$, then for sufficiently small $\varepsilon>0$ (such that $\rho_{n}(\varepsilon)<|\Sigma'|$), we have
\begin{align*}
|\nabla u(x)|\leq\frac{|\varphi(x', h(x'))-\varphi(0)|}{\varepsilon+dist^2(x',\Sigma')}+C|\widetilde{Q}[\varphi]|+C\|\varphi\|_{C^2(\partial D)},\quad x\in \Omega_{R_1},
\end{align*}
In particular, $|\nabla u|$ is bounded and there is no blowup occurring when $\varphi\equiv C$ in $\Omega_{R_1}$ (some positive constant).

(2) If $|\Sigma'|=0$ (that is, $\Sigma'=\{0'\}$), then we have
\begin{align*}
|\nabla u(x)|\leq \frac{C|\widetilde{Q}[\varphi]|}{\rho_n(\varepsilon)}+\frac{|\varphi(x', h(x'))-\varphi(0)|}{\varepsilon+|x'|^2}+C\|\varphi\|_{C^2(\partial D)},\quad x\in \Omega_{R_1},
\end{align*}
which shows that
\begin{align*}
|\nabla u(0',x_n)|\leq \frac{C|\widetilde{Q}[\varphi]|}{\rho_n(\varepsilon)}+C\|\varphi\|_{C^2(\partial D)},\quad 0<x_n<\varepsilon.
\end{align*}
From the proof of Theorem \ref{thm2},  if $|\widetilde{Q}[\varphi]|\geq c^*$ for some universal constant $c^*>0$, we also can obtain the lower bound of $|\nabla u(0',x_n)|$ as follows,
\begin{align*}
|\nabla u(0,x_n)|\geq \frac{|\widetilde{Q}[\varphi]|}{C\rho_n(\varepsilon)},\quad 0<x_n<\varepsilon.
\end{align*}
See the results in \cite{LX}.

(3) The relative convexity assumptions on $\partial D_{1}$ and $\partial D$ can also be weakened, which is similar to the discussions in Subsection \ref{general case} except replacing $\partial D_{2}$ by $\partial D$ and $h_{2}(x')$ by $h(x')$.
\end{remark}

Finally, from the view of methodology, the result of Theorem \ref{thm1} and Theorem \ref{thm2} can be extended to the general elliptic equations with a divergence form. For readers' convenience,  we here extend Theorem \ref{thm1} to such general elliptic equations. The analog of Theorem \ref{thm2} is left to the interested readers. Let $n,D_{1},D_{2},D,\varepsilon$ and $\varphi$ be the same as in Theorem \ref{thm1}, and let
$A_{ij}(x)\in C^{2}(\Omega)$ be $n\times{n}$ symmetric matrix functions and satisfy the uniform ellipticity condition
\begin{equation}\label{ene01a}
\lambda|\xi|^{2}\leq A_{ij}(x)\xi_{i}\xi_{j}\leq\Lambda|\xi|^{2},\quad\forall~\xi\in\mathbb R^{n}, \quad x\in \Omega,
\end{equation}
where $0<\lambda\leq\Lambda<+\infty$. We consider
\begin{equation}\label{equinftydiv}
\begin{cases}
\partial_{i}\left(A_{ij}(x)\partial_{j} u\right)=0,&\mbox{in}~\Omega,\\
u=C_{k},&\mbox{on}~\overline{D}_{k}, k=1,2,\\
\int_{\partial{D}_{k}}(A_{ij}(x)\partial_{j}u)\nu_{i}\big|_{+}=0,&k=1,2,\\
u=\varphi,&\mbox{on}~\partial{D}.
\end{cases}
\end{equation}
Then
\begin{theorem}\label{thm1.6}
Under the assumptions in Theorem \ref{thm1} for $D_{1}, D_{2}\subset{D}$. Let
$u\in{H}^{1}(D)\cap{C}^{1}(\overline{\Omega})$
be the solution to \eqref{equinftydiv}. Then the estimates \eqref{upperbound1'} and \eqref{upperbound10} also hold.
\end{theorem}

The rest of this paper is organized as follows. In Section \ref{sec_thm1}, we first decompose the solution $u$, and then use the energy method and the iteration technique initiated in \cite{llby,bll} to prove Theorem \ref{thm1}. A long remark is given Subsection \ref{general case} for more general inclusions $D_{1}$ and $D_{2}$. In Section \ref{extend}, we give the proof of Theorem \ref{thm2} for the boundary estimates. In Section \ref{sec_thm1.6}, the main ingredients to prove Theorem \ref{thm1.6} for general elliptic equations with a divergence form is listed.

\section{Proof of Theorem \ref{thm1}}\label{sec_thm1}

We decompose the solution $u(x)$ of \eqref{equinfty1} as follows
\begin{equation}\label{decomposition1_u}
u(x)=C_{1}v_{1}(x)+C_{2}v_{2}(x)+v_{3}(x),\qquad~x\in\,\Omega ,
\end{equation}
where $v_{j}\in{C}^{2}(\Omega)~(j=1,2,3)$, respectively, satisfying
\begin{equation}\label{equ_v1}
\begin{cases}
\Delta{v}_{1}=0,&\mathrm{in}~\Omega,\\
v_{1}=1,&\mathrm{on}~\partial{D}_{1},\\
v_{1}=0,&\mathrm{on}~\partial{D_{2}}\cup\partial{D},
\end{cases}
\end{equation}
\begin{equation}\label{equ_v2}
\begin{cases}
\Delta{v}_{2}=0,&\mathrm{in}~\Omega,\\
v_{2}=1,&\mathrm{on}~\partial{D}_{2},\\
v_{2}=0,&\mathrm{on}~\partial{D_{1}}\cup\partial{D},
\end{cases}
\end{equation}
\begin{equation}\label{equ_v3}
\begin{cases}
\Delta{v}_{3}=0,&\mathrm{in}~\Omega,\\
v_{3}=0,&\mathrm{on}~\partial{D}_{1}\cup\partial{D_{2}},\\
v_{3}=\varphi,&\mathrm{on}~\partial{D}.
\end{cases}
\end{equation}
Then by \eqref{decomposition1_u}, we have
\begin{align}\label{decomposition_u2}
\nabla{u}(x)=&C_{1}\nabla{v}_{1}(x)+C_{2}\nabla{v}_{2}(x)+\nabla{v}_{3}(x)\nonumber\\
=&(C_{1}-C_{2})\nabla{v}_{1}(x)+C_{2}\nabla({v}_{1}+{v}_{2})(x)+\nabla{v}_{3}(x),\qquad~x\in\,\Omega.
\end{align}
We need to estimate $|\nabla v_1|,\ |\nabla (v_1+v_2)|,\ |\nabla v_3|,\ |C_i| (i=1,2)$ and $|C_1-C_2|$ in order.

\subsection{Outline of the proof of Theorem \ref{thm1}}\label{proof of thm1}

In order to estimate $|\nabla v_j|,\ j=1, 2, 3$, we introduce an auxiliary function $\bar{u}_{1}\in{C}^{2}(\mathbb{R}^{n})$, such that $\bar{u}_{1}=1$ on $\partial{D}_{1}$, $\bar{u}_{1}=0$ on $\partial{D}_{2}\cup\partial D$,
\begin{align*}
\bar{u}_{1}(x', x_n)
=\frac{x_{n}-h_{2}(x')}{\varepsilon+h_{1}(x')-h_{2}(x')},\ \ \qquad~(x', x_n)\in\,\Omega_{R_{1}},
\end{align*}
and
\begin{equation*}
\|\bar{u}_{1}\|_{C^{2}(\Omega\setminus \Omega_{R_{1}})}\leq\,C.
\end{equation*}
Similarly, we define $\bar{u}_{2}=1$ on $\partial{D}_{2}$, $\bar{u}_{2}=0$ on $\partial{D}_{1}\cup\partial D$, $\bar{u}_{2}=1-\bar{u}_{1}$ in $\Omega_{R_{1}}$, and $\|\bar{u}_{2}\|_{C^{2}(\Omega\setminus \Omega_{R_{1}})}\leq\,C$.

Denote
\begin{equation*}
\delta(x'):=\varepsilon+h_{1}(x')-h_{2}(x'),\qquad\forall~(x',x_{n})\in\Omega_{R_{1}},
\end{equation*}
and
$$d(x'):=d_{\Sigma'}(x')=dist(x',\Sigma').$$
Using the assumptions on $h_{1}$ and $h_{2}$, \eqref{h1h2}--\eqref{h1h3}, a direct calculation gives
\begin{equation*}
\partial_{x_{i}}\bar{u}_{1}(x)=0,\quad\,i=1,\cdots,n-1,\quad
\partial_{x_{n}}\bar{u}_{1}(x)=\frac{1}{\varepsilon},\quad~x\in\Sigma,
\end{equation*}
and
\begin{equation*}
\left|\partial_{x_{i}}\bar{u}_{1}(x)\right|\leq\frac{Cd(x')}{\varepsilon+d^{2}(x')},\quad\,i=1,\cdots,n-1,\quad
\partial_{x_{n}}\bar{u}_{1}(x)=\frac{1}{\delta(x')},\quad~x\in\Omega_{R_{1}}\setminus\Sigma,
\end{equation*}
where $C$ is a {\it universal constant}, independent of $|\Sigma'|$. Therefore,
\begin{equation}\label{nablau_bar1}
\frac{1}{C\left(\varepsilon+d^{2}(x')\right)}\leq|\nabla\bar{u}_{i}(x)|\leq\,\frac{C}{\varepsilon+d^{2}(x')},\qquad\,x\in\Omega_{R_{1}},~i=1,2,
\end{equation}

\begin{prop}\label{prop1}
Assume the above, let $v_{1}, v_{2}, v_{3}\in{H}^1(\Omega)$ be the
weak solutions of \eqref{equ_v1}, \eqref{equ_v2} and \eqref{equ_v3}, respectively. Then, there exists some universal constant $\varepsilon_{1}>0$, such that for $0<\varepsilon<\varepsilon_{1}$, we have
\begin{equation}\label{nabla_w_i0}
\|\nabla(v_{i}-\bar{u}_{i})\|_{L^{\infty}(\Omega)}\leq\,C,\quad i=1,2,
\end{equation}
consequently, by using \eqref{nablau_bar1},
\begin{equation}\label{v1-bounded1}
\frac{1}{C\left(\varepsilon+d^{2}(x')\right)}\leq|\nabla v_{i}(x)|\leq\,\frac{C}{\varepsilon+d^{2}(x')},\qquad\,x\in\Omega_{R_{1}},~i=1,2.
\end{equation}
\begin{equation}\label{v1--bounded1}
|\nabla v_{i}(x)|\leq\,C,\qquad\,x\in\Omega\setminus\Omega_{R_{1}},~i=1,2;
\end{equation}
and
\begin{equation}\label{v1+v2_bounded1}
|\nabla(v_{1}+v_{2})(x)|\leq\,C,\qquad\,x\in\Omega;
\end{equation}
\begin{equation}\label{nabla_v3}
|\nabla{v}_{3}(x)|\leq C||\varphi||_{L^{\infty}(\partial D)},\quad\,x\in\Omega;
\end{equation}
where $C$ is a {\it universal constant}, independent of $|\Sigma'|$.
\end{prop}

Define
\begin{equation}\label{def_a_ij}
a_{ij}:=\int_{\partial D_{i}}\frac{\partial v_{j}}{\partial\nu},\quad b_{i}:=\int_{\partial D_{i}}\frac{\partial v_{3}}{\partial\nu},\quad i,j=1,2.
\end{equation}
By the third line of \eqref{equinfty1}, $C_1$ and $C_2$ satisfy
\begin{align}\label{system}
\begin{cases}
C_{1}a_{11}+C_{2}a_{12}+b_{1}=0,\\
C_{1}a_{21}+C_{2}a_{22}+b_{2}=0.
\end{cases}
\end{align}

Similarly as Lemma 2.4 in \cite{bly1}, we still have the following estimates for this general case. The proof is very similar with that of Lemma 2.4 in \cite{bly1}. We omit it.
\begin{lemma}\label{lemma aij}
\begin{align*}
a_{12}=a_{21}>0,\ a_{11}<0,\ a_{22}<0,
\end{align*}
\begin{align}\label{a11+a12}
-C\leq a_{11}+a_{21}\leq-\frac{1}{C},\ \  -C\leq a_{22}+a_{12}\leq -\frac{1}{C},\end{align}
and \begin{align}\label{b1-b2}
|b_1|\leq C \|\varphi\|_{C^2(\partial D)},\quad |b_2|\leq C\|\varphi\|_{C^2(\partial D)}.
\end{align}
\end{lemma}

However, the following is the main difference with the analog in \cite{bly1}. It will play a key role in the blowup analysis of $|\nabla u|$.

\begin{lemma}\label{lemma_a11}
For $n\geq2$, there exists some constant $\varepsilon_{2}\leq\varepsilon_{1}$,
such that for $0<\varepsilon<\varepsilon_{2}$, we have
\begin{align}\label{aii}
\frac{|\Sigma'|+\rho_{n}(\varepsilon)}{C\varepsilon}\leq-a_{ii}
\leq\frac{C\left(|\Sigma'|+\rho_{n}(\varepsilon)\right)}{\varepsilon},\quad i=1,2,
\end{align}
where $C$ is a {\it universal constant} but independent of $|\Sigma'|$.
\end{lemma}

\begin{remark}\label{rem aii}
Lemma \ref{lemma_a11} shows that

$(i)$ if $|\Sigma'|>0$, set $\varepsilon_{0}:=\min\{|\Sigma'|^{2},\varepsilon_{2}\}$, such that $\rho_{n}(\varepsilon)<|\Sigma'|$ if $0<\varepsilon<\varepsilon_{0}$, then for $0<\varepsilon<\varepsilon_{0}$,
\begin{equation}
\frac{|\Sigma'|}{C\varepsilon}\leq-a_{ii}\leq\frac{C|\Sigma'|}{\varepsilon},\quad i=1,2;
\end{equation}

$(ii)$ if $|\Sigma'|=0$ (that is, $\Sigma'=\{0'\}$), then we have
$$\frac{\rho_{n}(\varepsilon)}{C\varepsilon}\leq-a_{ii}\leq\frac{C\rho_{n}(\varepsilon)}{\varepsilon},$$
the same as Lemma 2.5--2.7 in \cite{bly1}, which leads the electric field to blow up, see the main results of \cite{bly1}.
\end{remark}

The proof of Lemma \ref{lemma_a11} is given in Subsection \ref{subsec3}. Thus, we are now in position to prove Theorem 1.1.

\begin{proof}[{\bf Proof of Theorem \ref{thm1}.}]
We need only to discuss the case that  $0<\varepsilon<\varepsilon_{0}$.
Since  $\|u\|_{H^{1}(\Omega)}\leq\,C$ (independent of $\varepsilon$), it follows from the trace embedding theorem that
\begin{equation}\label{C1C2}
|C_{1}|+|C_{2}|\leq\,C.
\end{equation}
By \eqref{system} and Lemma \ref{lemma aij}, we  have $a_{11}<0$, $a_{12}>0$, $a_{11}+a_{12}<0$, and $a_{11}a_{22}-a_{12}a_{21}>0$,
\begin{align}\label{C1-C1}
|C_{1}-C_{2}|=\frac{|b_{1}-\alpha b_{2}|}{|a_{11}-\alpha a_{12}|}=\frac{1}{|a_{11}|}\cdot\frac{|a_{11}||b_{1}-\alpha b_{2}|}{|a_{11}-\alpha a_{12}|},
\end{align}
where, by \eqref{a11+a12},
\begin{align}\label{alpha}
\frac{1}{C}\leq\alpha=\frac{a_{11}+a_{12}}{a_{21}+a_{22}}\leq C,
\end{align}
for some positive constant $C$.
Thus, we have
$$|a_{11}|<|a_{11}-\alpha a_{12}|<(1+\alpha)|a_{11}|\leq C|a_{11}|,$$
that is,
\begin{align}\label{a11 a22}
\frac{1}{C}\leq\frac{|a_{11}|}{|a_{11}-\alpha a_{12}|}
\leq C.
\end{align}
It follows from \eqref{b1-b2} and \eqref{alpha} that
\begin{align}\label{b1 2}
|b_1-\alpha b_2|\leq C\|\varphi\|_{C^2(\partial D)}.
\end{align}
By \eqref{C1-C1}, \eqref{a11 a22}, \eqref{b1 2} and Lemma \ref{lemma_a11}, if $|\Sigma'|>0$, we have
\begin{align}\label{c1--c2}
|C_1-C_2|\leq\frac{C\|\varphi\|_{C^2(\partial D)}}{|a_{11}|}\leq \frac{C\varepsilon}{|\Sigma'|+\rho_{n}(\varepsilon)}\|\varphi\|_{C^2(\partial D)},\quad\mathrm{for}\ 0<\varepsilon<\varepsilon_{0}.
\end{align}
Recalling \eqref{decomposition_u2}, using \eqref{v1-bounded1}-\eqref{nabla_v3}, \eqref{C1C2} and \eqref{c1--c2}, we obtain for $x\in\Omega_{R_{1}}$,
\begin{align*}
|\nabla{u}(x)|\leq&|C_{1}-C_{2}||\nabla{v}_{1}(x)|+|C_{2}||\nabla({v}_{1}+{v}_{2})(x)|+|\nabla{v}_{3}(x)|\\
\leq&\frac{C\varepsilon}{|\Sigma'|+\rho_{n}(\varepsilon)}\frac{1}{\varepsilon+d^{2}(x')}\|\varphi\|_{C^2(\partial D)}+C\|\varphi\|_{C^2(\partial D)},
\end{align*}
and \eqref{upperbound10}. The proof of Theorem \ref{thm1} is completed.
\end{proof}

\begin{remark}\label{rem2.5}
From the proof of Theorem \ref{thm1}, we can see that if $|\Sigma'|=0$ (that is, $\Sigma'=\{0'\}$), we have the pointwise upper bound estimates
$$|\nabla{u}(x',x_{n})|\leq \frac{C\varepsilon}{\rho_{n}(\varepsilon)}\frac{1}{\varepsilon+|x'|^{2}}||\varphi||_{C^{2}(\partial D)},\qquad\mbox{for}~x\in\Omega_{R_{1}}.
$$
Especially,
$$|\nabla{u}(0',x_{n})|\leq \frac{C}{\rho_{n}(\varepsilon)}||\varphi||_{C^{2}(\partial D)},\quad 0<x_{n}<\varepsilon.
$$
Actually, by \eqref{decomposition_u2}, \eqref{C1-C1}, and \eqref{a11 a22}, we have
\begin{align*}
|\nabla{u}(0',x_{n})|\geq&|C_{1}-C_{2}||\nabla{v}_{1}(0',x_{n})|-|C_{2}||\nabla({v}_{1}+{v}_{2})(0',x_{n})|-|\nabla{v}_{3}(0',x_{n})|\\
\geq&|b_{1}-\alpha b_{2}|\cdot\frac{|a_{11}|}{|a_{11}-\alpha a_{12}|}\frac{1}{|a_{11}|}|\nabla{v}_{1}(0',x_{n})|-C\\
\geq&|b_{1}-\alpha b_{2}|\cdot\frac{1}{C\rho_{n}(\varepsilon)}-C.
\end{align*}
Thus, if $|b_{1}-\alpha b_{2}|=constant\neq 0$, then the blowup must occur. Recalling the definition of $a_{ij}$ and $b_{i}$, we have
$$a_{11}+a_{12}=a_{11}+a_{21}=\int_{\partial D_{1}}\frac{\partial v_{1}}{\partial\nu}+\int_{\partial D_{2}}\frac{\partial v_{1}}{\partial\nu}=\int_{\partial D}\frac{\partial v_{1}}{\partial\nu},$$
and
$$a_{22}+a_{21}=a_{22}+a_{12}=\int_{\partial D_{2}}\frac{\partial v_{2}}{\partial\nu}+\int_{\partial D_{1}}\frac{\partial v_{2}}{\partial\nu}=\int_{\partial D}\frac{\partial v_{2}}{\partial\nu}.$$
So that, using \eqref{alpha},
\begin{align*}
(b_{1}-\alpha b_{2})(a_{22}+a_{21})=&\left((a_{22}+a_{21})b_{1}-(a_{11}+a_{12})b_{2}\right)\\
=&\left(\int_{\partial D}\frac{\partial v_{2}}{\partial\nu}\int_{\partial D_{1}}\frac{\partial v_{3}}{\partial\nu}-\int_{\partial D}\frac{\partial v_{1}}{\partial\nu}\int_{\partial D_{2}}\frac{\partial v_{3}}{\partial\nu}\right),
\end{align*}
which is the linear functional of $\varphi$, $Q_{\varepsilon}[\varphi]$, defined in \cite{bly1}. If there exists an $\varphi_{0}$ such that its limit functional $Q^{*}_{\varepsilon}[\varphi_{0}]\neq0$, then $Q_{\varepsilon}[\varphi_{0}]\neq0$ too, for sufficienty small $\varepsilon$. More details can be referred to Section 3 in \cite{bly1}.
\end{remark}

\subsection{Proof of Proposition \ref{prop1}}\label{subsec_prop1}

In order to show the role of $\Sigma$, we give a proof with some details and list the main difference, although the main idea is in spirit from \cite{bll, llby}. We emphasize that in this subsection the constant $C$ is independent of $|\Sigma'|$.
\begin{proof}
\noindent{\bf STEP 1.} Proof of \eqref{nabla_w_i0}. We prove it for $i=1$, and $i=2$ is the same. Denote
\begin{equation}\label{def_w}
w_{1}:=v_{1}-\bar{u}_{1}.
\end{equation}
By the definition of $v_{1}$, \eqref{equ_v1}, and using \eqref{def_w}, we have
\begin{equation}\label{w20}
\left\{
\begin{aligned}
-\Delta{w}_{1}&=\Delta\bar{u}_{1}\quad\mbox{in}~\Omega,\\
w_{1}&=0\quad\ \ \mbox{on}~\partial \Omega.
\end{aligned}\right.
\end{equation}
Since
\begin{equation}\label{estimate_baru}
|\bar{u}_{1}|+|\nabla\bar{u}_{1}|+|\nabla^{2}\bar{u}_{1}|\leq\,C,\quad\mbox{in}~~ \Omega\setminus\Omega_{R_1/2},
\end{equation}
by the standard elliptic theory, we know that
\begin{equation}\label{nabla_w_out}
|w_{1}|+\left|\nabla{w}_{1}\right|\leq\,C,
\quad\mbox{in}~~ \Omega\setminus\Omega_{R_{1}}.
\end{equation}
Therefore, in order to show \eqref{nabla_w_i0}, we only need to prove
\begin{equation}\label{section2 energyw1}
\left\|\nabla{w}_{1}\right\|_{L^{\infty}(\Omega_{R_{1}})}\leq\,C,
\end{equation}
We divide into three steps to prove \eqref{section2 energyw1}.

{\bf STEP 1.1.} Proof of boundedness of the energy of $w_{1}$ in $\Omega$, that is,
\begin{equation}\label{1energy_w}
\int_{\Omega}\left|\nabla{w}_{1}\right|^{2}\leq\,C.
\end{equation}

Using the maximum principle, we have $0<v_{1}<1$ in $\Omega$, so that
\begin{equation}\label{w_bdd}
\|w_{1}\|_{L^{\infty}(\Omega)}\leq\,C.
\end{equation}
By a direct computation,
\begin{equation}\label{nabla2u_bar}
\Delta\bar{u}_{1}(x)=0,\quad x\in\Sigma,\quad |\Delta\bar{u}_{1}(x)|\leq\frac{C}{\varepsilon+d^{2}(x')},\quad\,x\in\Omega_{R_{1}}\setminus\Sigma.
\end{equation}
Multiplying the equation in \eqref{w20} by $w_{1}$ and integrating by parts, it follows from \eqref{estimate_baru} and \eqref{nabla2u_bar} that
\begin{align*}
\int_{\Omega}|\nabla{w}_{1}|^{2}
=\int_{\Omega}w_{1}\left(\Delta\bar{u}_{1}\right)
\leq\,\|w_{1}\|_{L^{\infty}(\Omega)}\left(\int_{\Omega_{R_{1}}\setminus\Sigma}|\Delta\bar{u}_{1}|+C\right)
\leq\,C.
\end{align*}
So \eqref{1energy_w} is proved.

{\bf STEP 1.2.} Proof of
\begin{equation}\label{energy_w_inomega_z1}
\int_{\widehat{\Omega}_{\delta}(z')}\left|\nabla{w}_{1}\right|^{2}dx\leq
C\delta^{n},
\end{equation}
where $\delta=\delta(z')=\varepsilon+h_{1}(z')-h_{2}(z')$, and
\begin{equation*}
\widehat{\Omega}_{t}(z'):=\left\{x\in \mathbb{R}^{n}~\big|~h_{2}(x')<x_{n}<\varepsilon+h_{1}(x'),~|x'-z'|<{t}\right\}.
\end{equation*}

The following iteration scheme we used is similar in spirit to that used in  \cite{bll,llby}. For $0<t<s<R_{1}$, let $\eta$ be a smooth cutoff function satisfying $\eta(x')=1$ if $|x'-z'|<t$, $\eta(x')=0$ if $|x'-z'|>s$, $0\leq\eta(x')\leq1$ if $t\leq|x'-z'|\leq\,s$, and $|\nabla_{x'}\eta(x')|\leq\frac{2}{s-t}$. Multiplying the equation in \eqref{w20} by $w_{1}\eta^{2}$ and integrating by parts
leads  to the Caccioppolli-type inequality
\begin{align}\label{FsFt11}
\int_{\widehat{\Omega}_{t}(z')}|\nabla{w}_{1}|^{2}\leq\,\frac{C}{(s-t)^{2}}\int_{\widehat{\Omega}_{s}(z')}|w_{1}|^{2}
+(s-t)^{2}\int_{\widehat{\Omega}_{s}(z')}\left|\Delta\bar{u}_{1}\right|^{2}.
\end{align}
We further divide into three cases to derive the iteration formula by using \eqref{FsFt11}.

{\bf Case 1.} For $z'\in\Sigma'_{-\sqrt{\varepsilon}}:=\{x'\in\Sigma'~|~ d(x',\partial\Sigma')>\sqrt{\varepsilon}\}$ and $0<s<\sqrt{\varepsilon}$, where $\delta(z')=\varepsilon$. We here assume that $B'_{\sqrt{\varepsilon}}\subset\Sigma'$ (otherwise, start from Case 2), then
\begin{align}\label{int_w1}
\int_{\widehat{\Omega}_{s}(z')}|w_{1}|^{2}&=\int_{|x'-z'|<s}\int_{0}^{\varepsilon}\left(\int_{0}^{x_{n}}\partial_{ x_{n}}w_{1}dx_{n}\right)^{2}dx_{n}dx'\nonumber\\
&\leq\int_{|x'-z'|<s}\varepsilon^{2}\int_{0}^{\varepsilon}|\partial_{x_{n}}w_{1}|^{2}dx_{n}dx'\nonumber\\
&\leq C\varepsilon^{2}\int_{\widehat{\Omega}_{s}(z')}|\nabla{w}_{1}|^{2}.
\end{align}
Denote
$$F(t):=\int_{\widehat{\Omega}_{t}(z')}|\nabla{w}_{1}|^{2}.$$
It follows from \eqref{nabla2u_bar}, \eqref{FsFt11} and \eqref{int_w1} that
\begin{equation}\label{energy_w1}
F(t)\leq\left(\frac{c_{1}\varepsilon}{s-t}\right)^{2}F(s),
\end{equation}
where $c_{1}$ is a {\it universal constant} but independent of $|\Sigma'|$.

Let $k=\left[\frac{1}{4c_{1}\sqrt{\varepsilon}}\right]$ and $t_{i}=\delta+2c_{1}i\varepsilon, i=0,1,2,\cdots,k$. Then by \eqref{energy_w1} with $s=t_{i+1}$ and $t=t_{i}$, we have
$$F(t_{i})\leq\frac{1}{4}F(t_{i+1}).$$
After $k$ iterations, using \eqref{1energy_w}, we have
$$F(t_{0})\leq(\frac{1}{4})^{k}F(t_{k})\leq C\varepsilon^{n}.$$
Therefore, for some sufficiently small  $\varepsilon>0$,
\begin{equation*}
\int_{\widehat{\Omega}_{\delta}(z')}\left|\nabla{w}_{1}\right|^{2}dx\leq C\varepsilon^{n}.
\end{equation*}

{\bf Case 2.} For $z'\in\Sigma'_{\sqrt{\varepsilon}}\setminus\Sigma'_{-\sqrt{\varepsilon}}$, where $\Sigma'_{\sqrt{\varepsilon}}:=\{x'\in B'_{R_{1}}~| ~dist(x', \Sigma')<\sqrt{\varepsilon}\}$ and $0<s<\sqrt{\varepsilon}$, we have $\varepsilon\leq\delta(z')\leq C\varepsilon$. Using \eqref{nabla2u_bar}, we have

\begin{equation}\label{integal_Lubar11_in}
\int_{\widehat{\Omega}_{s}(z')}
\left|\Delta\bar{u}_{1}\right|^{2}
\leq\frac{Cs^{n-1}}{\varepsilon}.
\end{equation}
Note that
\begin{align}\label{energy_w_square_in}
\int_{\widehat{\Omega}_{s}(z')}|w_{1}|^{2}
\leq\,C\varepsilon^{2}\int_{\widehat{\Omega}_{s}(z')}|\nabla{w}_{1}|^{2}.
\end{align}
It follows from \eqref{FsFt11}, \eqref{integal_Lubar11_in} and
\eqref{energy_w_square_in} that
\begin{equation}\label{tildeF111_in}
F(t)\leq\,\left(\frac{c_{2}\varepsilon}{s-t}\right)^{2}F(s)+C(s-t)^{2}\frac{s^{n-1}}{\varepsilon},
\quad\forall~0<t<s<\sqrt{\varepsilon},
\end{equation}
where $c_2$ is another {\it universal constant} but independent of $|\Sigma'|$.

Let $k=\left[\frac{1}{4c_{2}\sqrt{\varepsilon}}\right]$ and $t_{i}=\delta+2c_{2}i\varepsilon$, $i=0,1,2,\cdots,k$. Then
by \eqref{tildeF111_in} with $s=t_{i+1}$ and $t=t_{i}$, we have
$$F(t_{i})\leq\,\frac{1}{4}F(t_{i+1})
+C(i+1)^{n-1}\varepsilon^{n}.$$
After $k$ iterations,
using (\ref{1energy_w}), we have
\begin{eqnarray*}
F(t_{0})
\leq (\frac{1}{4})^{k}F(t_{k})
+C\varepsilon^{n}\sum_{i=0}^{k-1}(\frac{1}{4})^{i}(i+1)^{n-1}
\leq\,C\varepsilon^{n}.
\end{eqnarray*}
Therefore, for some sufficiently small   $\varepsilon>0$, we have
$$\int_{\widehat{\Omega}_{\delta}(z')}|\nabla{w}_{1}|^{2}\leq\,C\varepsilon^{n}.$$

{\bf Case 3.} For $z'\in B'_{R_{1}}\setminus\Sigma'_{\sqrt{\varepsilon}}$, and $0<s<\frac{2}{3}d(z')$, we have $Cd^{2}(z')\leq\delta(z')\leq(C+1)d^{2}(z')$.
Estimate \eqref{integal_Lubar11_in} and \eqref{energy_w_square_in} become
\begin{equation*}\label{integal_Lubar11}
\int_{\widehat{\Omega}_{s}(z')}\left|\Delta\bar{u}_{1}\right|^{2}\leq\frac{Cs^{n-1}}{d^{2}(z')}.
\end{equation*}
and
\begin{align*}\label{energy_w_square}
\int_{\widehat{\Omega}_{s}(z')}|w_{1}|^{2}
\leq&\,Cd^{4}(z')\int_{\widehat{\Omega}_{s}(z')}|\nabla{w}_{1}|^{2},
\end{align*} respectively. Furthermore, in view of \eqref{FsFt11},
estimate \eqref{tildeF111_in} becomes
\begin{equation}\label{tildeF111}
F(t)\leq\,\left(\frac{c_{3}d^{2}(z')}{s-t}\right)^{2}F(s)+C(s-t)^{2}\frac{s^{n-1}}{d^{2}(z')},
\quad\forall~0<t<s<\frac{2}{3}d(z'),
\end{equation}
where $c_3$ is another {\it universal constant} but independent of $|\Sigma'|$.

Let $k=\left[\frac{1}{4c_{3}d(z')}\right]$ and $t_{i}=\delta+2c_{3}i\,d^{2}(z')$, $i=0,1,2,\cdots,k$. Then applying \eqref{tildeF111} with $s=t_{i+1}$ and $t=t_{i}$, we have
$$F(t_{i})\leq\,\frac{1}{4}F(t_{i+1})+C(i+1)^{n-1}d^{2n}(z'),
$$
After $k$ iterations, using (\ref{1energy_w}), we have for some sufficiently small   $\varepsilon>0$,
\begin{eqnarray*}
F(t_{0}) \leq (\frac{1}{4})^{k}F(t_{k})+Cd^{2n}(z')\sum_{i=0}^{k-1}(\frac{1}{4})^{i}(i+1)^{n-1}
\leq Cd^{2n}(z').
\end{eqnarray*}
This implies that
$$\int_{\widehat{\Omega}_{\delta}(z')}|\nabla{w}_{1}|^{2}\leq\,Cd^{2n}(z').$$
Therefore, \eqref{energy_w_inomega_z1} is proved.

{\bf STEP 1.3.}  Rescaling and $L^{\infty}$ estimates. Making a change of variables
\begin{equation*}
 \left\{
  \begin{aligned}
  &x'-z'=\delta y',\\
  &x_n=\delta y_n,
  \end{aligned}
 \right.
\end{equation*}
then $\widehat{\Omega}_{\delta}(z')$ becomes $Q_{1}$ of nearly unit size, where
$$Q_{r}=\left\{y\in\mathbb{R}^{n}~\Big|~\frac{1}{\delta}h_{2}(\delta{y}'+z')<y_{n}
<\frac{\varepsilon}{\delta}+\frac{1}{\delta}h_{1}(\delta{y}'+z'),~|y'|<r\right\},\quad\mbox{for}~~r\leq1,$$ and the top and
bottom boundaries
become
$$
y_n=\hat{h}_{1}(y'):=\frac{1}{\delta}
\left(\varepsilon+h_{1}(\delta\,y'+z')\right),\quad|y'|<1,$$
and
$$y_n=\hat{h}_{2}(y'):=\frac{1}{\delta}h_{2}(\delta\,y'+z'), \quad |y'|<1.
$$
By the standard bootstrap argument of $W^{2,p}$ estimates for elliptic equations in unit size domain, the same as in the step 1.3 of \cite{LX}, we obtain
\begin{equation}
\left\|\nabla{w_{1}}\right\|_{L^{\infty}(\widehat{\Omega}_{\delta/2}(z'))}\leq\,
\frac{C}{\delta}\left(\delta^{1-\frac{n}{2}}\left\|\nabla{w_{1}}\right\|_{L^{2}(\widehat{\Omega}_{\delta}(z'))}
+\delta^{2}\left\|\Delta\bar{u}_{1}\right\|_{L^{\infty}(\widehat{\Omega}_{\delta}(z'))}\right).
\label{AAA}
\end{equation}

{\bf Case 1.} For $z'\in\Sigma'_{-\sqrt{\varepsilon}}$. In this case, $\Delta\bar{u}_{1}=0$. It follows from \eqref{AAA} and \eqref{energy_w_inomega_z1} that
$$\left|\nabla{w}_{1}(z',z_{n})\right|\leq\frac{C\delta^{1-\frac{n}{2}}\delta^{\frac{n}{2}}}{\delta}
=\,C,
\qquad\forall \ 0<z_{n}<\varepsilon.$$

{\bf Case 2.} For $z'\in\Sigma'_{\sqrt{\varepsilon}}\setminus\Sigma'_{-\sqrt{\varepsilon}}$.
Using \eqref{nabla2u_bar} and $\varepsilon\leq\delta(z')\leq C\varepsilon$,
$$\delta^{2}\left|\Delta\bar{u}_{1}\right|\leq
\frac{C \delta^{2}}{\varepsilon}\leq\,C
\delta, \qquad
\mbox{in}\
\widehat \Omega_\delta(z').$$
It follows from (\ref{AAA}) and (\ref{energy_w_inomega_z1}) that
$$\left|\nabla{w}_{1}(z',z_{n})\right|\leq\frac{C(\delta^{1-\frac{n}{2}}\delta^{\frac{n}{2}}+\delta)}{\delta}
\leq\,C,
\qquad\forall \ h_{2}(z')<z_{n}<\varepsilon+h_1(z').$$

{\bf Case 3.} For $z'\in B'_{R_{1}}\setminus\Sigma'_{\sqrt{\varepsilon}}$.
Using \eqref{nabla2u_bar} and $Cd^{2}(z')\leq\delta(z')\leq(C+1)d^{2}(z')$,
$$\delta^{2}\left|\Delta\bar{u}_{1}\right|\leq\frac{C\delta^{2}}{d^{2}(z')}\leq\,C\delta,
\qquad \mbox{in}\
\widehat \Omega_\delta(z').$$
We deduce from (\ref{AAA}) and (\ref{energy_w_inomega_z1}) that
$$\left|\nabla{w}_{1}(z',z_{n})\right|\leq\frac{C\Big(\delta^{1-\frac{n}{2}}\delta^{\frac{n}{2}}+\delta\Big)}{\delta}\leq\,C,
\qquad\forall \
h_{2}(z')<z_{n}<\varepsilon+h_1(z').$$
Estimate \eqref{nabla_w_i0} is established.

\noindent{\bf STEP 2.} Proof of \eqref{v1+v2_bounded1} and \eqref{nabla_v3}.
Recalling the definitions of $v_{1}$ and $v_{2}$, \eqref{equ_v1} and \eqref{equ_v2}, we have
\begin{align*}
\begin{cases}
\Delta(v_{1}+v_{2}-1)=0,&\mbox{in}~\Omega,\\
v_{1}+v_{2}-1=0,&\mbox{on}~\partial{D}_{i},~i=1,2,\\
v_{1}+v_{2}-1=-1,&\mbox{on}~\partial{D}.
\end{cases}
\end{align*}
Therefore, the result of \cite{llby} shows that
$$|\nabla (v_{1}+ v_{2})|\leq C\quad \hbox{in}\ \Omega.$$
By the same reason, we have
$$||\nabla v_{3}||_{L^{\infty}(\Omega)}\leq C||\varphi||_{L^{\infty}(\partial D)}.$$
The proof of Propostion \ref{prop1} is finished.
\end{proof}

\subsection{Proof of Lemma \ref{lemma_a11}}\label{subsec3}
In order to prove Lemma \ref{lemma_a11}, we need the following well-known property for bounded convex domains, see e.g. Theorem 1.8.2 in \cite{gu}.

\begin{lemma}
If $D\subset\mathbb{R}^{n}$ is a bounded convex set with nonempty interior and $E$ is the ellipsoid of minimum volume containing $D$ center at the center of mass of $D$, then
$$n^{-3/2}E\subset D\subset E,$$
where $aE$ denotes the $a$-dilation of $E$ with respect to its center.
\end{lemma}

Thus, for bounded convex $(n-1)$-dimensional domain $\Sigma'$, there exists a $E'$ such that
$$(n-1)^{-3/2}E'\subset \Sigma'\subset E'.$$
Denote the length of the longest principal semi-axis as $R_{0}$ and the length of the shortest principal semi-axis as $\widetilde{R}_{0}>0$. In order to show the key role of $|\Sigma'|$ in the blowup analysis of $|\nabla u|$, for simplicity, we suppose that $\frac{R_{0}}{\widetilde{R}_{0}}\geq a$ for some constant $a>0$. Set $r_{0}=(n-1)^{-3/2}\widetilde{R}_{0}$. Obviously, $B'_{r_{0}}\subset \Sigma'\subset E'\subset B'_{R_{0}}$. Then, there exists a constant $C$, depending only on $n$ and $a$, such that
\begin{equation}\label{sigma}
|B'_{R_{0}}|\leq C|\Sigma'|.
\end{equation}

\begin{proof}[Proof of Lemma \ref{lemma_a11}]
We here estimate $a_{11}$ for instance, since $a_{22}$ is the same. Recalling the definition of $a_{11}$, \eqref{def_a_ij}, and using Green's formula, we have
$$a_{11}=\int_{\partial D_{1}}\frac{\partial v_{1}}{\partial\nu}=\int_{\partial D_{1}}\frac{\partial v_{1}}{\partial\nu}v_{1}=-\int_{\Omega}|\nabla v_{1}|^{2}.$$
We decompose it into three parts,
\begin{align}\label{a11'}
-a_{11}=\int_{\Sigma}|\nabla v_{1}|^{2}+\int_{\Omega_{R_1}\setminus\Sigma}|\nabla v_{1}|^{2}+\int_{\Omega\setminus \Omega_{R_1}}|\nabla v_{1}|^{2}.
\end{align}
For the first term, by (\ref{v1-bounded1}), we have
\begin{align}\label{first_term}
\frac{|\Sigma'|}{C\varepsilon}\leq\int_{\Sigma'}\int_{0}^{\varepsilon}\frac{1}{C\varepsilon^2}dx_ndx'\leq\int_{\Sigma}|\nabla v_{1}|^{2}\leq\int_{\Sigma'}\int_{0}^{\varepsilon}\frac{C}{\varepsilon^2}dx_ndx'\leq\frac{C|\Sigma'|}{\varepsilon}.
\end{align}
For the last term of \eqref{a11'}, it is easy to see from (\ref{v1--bounded1}) that
\begin{align}\label{R3}
\int_{\Omega\setminus \Omega_{R_1}}|\nabla v_{1}|^{2}\leq C.
\end{align}

For the middle term of \eqref{a11'}, it is complicated a little bit. First, in view of (\ref{v1-bounded1}) again, we have
\begin{align*}
&\int_{B'_{R_{1}}\setminus\Sigma'}
\int_{h_2(x')}^{\varepsilon+h_1(x')}\frac{1}{C(\varepsilon+d^2(x'))^{2}}dx_ndx'\\
&\leq
 \int_{\Omega_{R_1}\setminus\Sigma}|\nabla v_{1}|^{2}\\
 &\leq\int_{B'_{R_{1}}\setminus\Sigma'}
\int_{h_2(x')}^{\varepsilon+h_1(x')}\frac{C}{(\varepsilon+d^2(x'))^2}dx_ndx',
\end{align*}
which implies that
\begin{align}\label{a11_mid}
\int_{B'_{R_{1}}\setminus\Sigma'}\frac{dx'}{C(\varepsilon+d^2(x'))}\leq
\int_{\Omega_{R_1}\setminus\Sigma}|\nabla v_{1}|^{2}\leq\int_{B'_{R_{1}}\setminus\Sigma'}\frac{Cdx'}{\varepsilon+d^2(x')}.
\end{align}
We divide into three cases by dimension to estimate \eqref{a11_mid} in the following.

If $n=2$, then $\Sigma'=(-R_{0},R_{0})$, and $d(x')=|x'|-R_{0}$. We can choose some constant $\tilde{\varepsilon}\in(0,1)$ depending only on $R_{1}$, such that for $0<\varepsilon<\tilde{\varepsilon}$,
\begin{align}\label{n=2}
\int_{R_0}^{R_1}\frac{dr}{C\left(\varepsilon+(r-R_0)^2\right)}=\frac{1}{C}\int_{0}^{R_1-R_0}\frac{dr}{\varepsilon+r^2}=\frac{1}{C\sqrt{\varepsilon}}\arctan\frac{R_{1}-R_{0}}{\sqrt{\varepsilon}}.
\end{align}
Inserting \eqref{R3}--\eqref{n=2} to \eqref{a11'}, we have, for sufficiently small $\varepsilon$ (say, at least less than $(R_{1}-R_{0})^{2}$),
\begin{align*}
\frac{1}{C}\left(\frac{|\Sigma'|}{\varepsilon}+\frac{1}{\sqrt{\varepsilon}}\right)\leq -a_{11}\leq C\left(\frac{|\Sigma'|}{\varepsilon}+\frac{1}{\sqrt{\varepsilon}}\right),
\end{align*}
which implies that \eqref{aii} holds for $n=2$.

If $n=3$, notice that \eqref{sigma}, then choosing some constant $\tilde{\varepsilon}_1\in(0,1/e)$, such that for $0<\varepsilon<\tilde{\varepsilon}_1$, we have
\begin{align*}
\int_{B'_{R_{1}}\setminus\Sigma'}\frac{dx'}{\varepsilon+d^2(x')}&\leq\int_{B'_{R_{1}}\setminus B'_{r_{0}}}\frac{dx'}{\varepsilon+dist^{2}(x',B'_{R_{0}})}\\
&\leq\int_{r_{0}}^{R_{0}}\frac{Cr}{\varepsilon}dr+\int_{R_{0}}^{R_{1}}\frac{Cr}{\varepsilon+(r-R_{0})^{2}}dr\\
&\leq\frac{C(R_0^2-r_{0}^{2})}{\varepsilon}+C\int_{R_{0}}^{R_{1}}\frac{r-R_{0}}{\varepsilon+(r-R_{0})^{2}}dr
+C\int_{R_{0}}^{R_{1}}\frac{R_{0}}{\varepsilon+(r-R_{0})^{2}}dr\\
&\leq C\left(\frac{R_0^2}{\varepsilon}+|\ln\varepsilon|+\frac{R_0}{\sqrt{\varepsilon}}\right)\\
&\leq C\left(|\ln\varepsilon|+\frac{|\Sigma'|}{\varepsilon}\right),
\end{align*}where the Cauchy inequality has been used in the last inequality.

On the other hand, we pick a point $p\in\partial\Sigma'$, take a quadrant $Q$ with $p$ as the vertex, $(R_{1}-R_{0})/2$ as the radius, and symmetric with the normal of $p$, denoted $N_{p}$. Then, in the polar coordinates $\{p; r, \theta\}$ with $p$ as the center, for $x'\in Q$, we have $x'=p+(r\cos\theta, r\sin\theta)$ and $dist(x',\Sigma')\leq dist(x',p)$. There exists some small positive constant $\tilde{\varepsilon}\in(0,\tilde{\varepsilon}_1)$, depending only on $R_{1}$, if $0<\varepsilon<\tilde{\varepsilon}$, we have
\begin{align*}
\int_{B'_{R_{1}}\setminus\Sigma'}\frac{dx'}{\varepsilon+d^2(x')}&\geq\int_{Q}\frac{dx'}{\varepsilon+dist^{2}(x',p)}\\
&=\int_{-\frac{\pi}{4}}^{\frac{\pi}{4}}\int_{0}^{\frac{R_{1}-R_{0}}{2}}\frac{rdr}{\varepsilon+r^{2}}\\
&\geq\frac{1}{C}|\ln\varepsilon|.
\end{align*}
Substituting these two estimates above, together with \eqref{first_term} and \eqref{R3}, into \eqref{a11'}, we have \eqref{aii} for $n=3$.

If $n\geq4$, similarly by using \eqref{sigma}, we have
\begin{align*}
\int_{B'_{R_{1}}\setminus\Sigma'}\frac{dx'}{\varepsilon+d^2(x')}&\leq
\int_{r_{0}}^{R_{0}}\frac{Cr^{n-2}}{\varepsilon}dr+\int_{R_{0}}^{R_{1}}\frac{Cr^{n-2}}{\varepsilon+(r-R_{0})^{2}}dr\\
&\leq\frac{C(R_0^{n-1}-r_{0}^{n-1})}{\varepsilon}+C\int_{0}^{R_1-R_0}\frac{(t+R_0)^{n-2}}{\varepsilon+t^2}dt\\
&\leq \frac{CR_0^{n-1}}{\varepsilon}+CR_0^{n-2}\int_{0}^{R_1-R_0}\frac{1}{\varepsilon+t^2}dt+C\int_{0}^{R_1-R_0}\frac{{t^{n-2}}}{\varepsilon+t^2}dt\\
&\leq C\left(\frac{R_0^{n-1}}{\varepsilon}+\frac{R_0^{n-2}}{\sqrt{\varepsilon}}+\int_{0}^{R_1-R_0}\frac{{t^{2}}}{\varepsilon+t^2}t^{n-4}dt\right),\\
&\leq C\left(\frac{|\Sigma'|}{\varepsilon}+1\right).
\end{align*}
For any $p\in\partial\Sigma'$, we also can construct a cone $Q\subset B'_{R_{1}}\setminus\Sigma'$ with $p$ as the vertex,  such that $dist(x',\Sigma')\leq dist(x',p)$ whenever $x'\in Q$. Then for sufficiently small $\varepsilon$,
\begin{align*}
\int_{B'_{R_{1}}\setminus\Sigma'}\frac{dx'}{\varepsilon+d^2(x')}&\geq\int_{Q}\frac{dx'}{\varepsilon+dist^{2}(x',p)}\\
&\geq\frac{1}{C}\int_{0}^{\frac{R_{1}-R_{0}}{2}}\frac{r^{n-2}}{\varepsilon+r^{2}}dr\\
&\geq\frac{1}{C}.
\end{align*}
Then, we have \eqref{aii} for $n\geq4$. The proof of Lemma \ref{lemma_a11} is completed.
\end{proof}

\subsection{More general $D_{1}$ and $D_{2}$}\label{general case}

We consider a somewhat more general setting: We assume that the domain $D_{1}^{0}$ and $D_{2}^{0}$ are convex  outside $\Sigma$ and with growth order $m$, $m\geq2$. Precisely, for $x'\in B'_{R_{1}}\setminus\overline{\Sigma'}$,
\begin{equation*}
\lambda_{0}d^{m}(x')\leq\,h_{1}(x')-h_{2}(x')\leq\lambda_{1}d^{m}(x'),
\end{equation*}
and
\begin{equation*}
|\nabla_{x'}h_{i}(x')|\leq\,Cd^{m-1}(x'),~|\nabla_{x'}^{2}h_{i}(x')|\leq\,Cd^{m-2}(x'),\quad i=1,2
\end{equation*}
for some $\varepsilon$-independent constants $0<\lambda_{0}<\lambda_{1}$. Clearly,
$$\frac{1}{C}(\varepsilon+d^{m}(x'))\leq\delta(x')\leq\,C(\varepsilon+d^{m}(x')).$$
Denote
\begin{align*}
\rho_n^m(\varepsilon)=\begin{cases}\varepsilon^{\frac{n-1}{m}},& m>n-1,\\
\varepsilon|\ln\varepsilon|,&m=n-1,\\
\varepsilon,&m<n-1.
\end{cases}
\end{align*}
By the iteration process, Proposition \ref{prop1} for estimates of $|\nabla v_{i}|$, $i=1,2$, also hold except replacing \eqref{v1-bounded1} by
\begin{equation}\label{v1-bounded1'}
\frac{1}{C\left(\varepsilon+d^{m}(x')\right)}\leq|\nabla v_{i}(x)|\leq\,\frac{C}{\varepsilon+d^{m}(x')},\qquad\,x\in\Omega_{R_{1}}\setminus\Sigma,~i=1,2.
\end{equation}
For readers' convenience, we give the proof of the estimates of $a_{ii}$ $(i=1,2)$ in this general setting.
\begin{lemma}\label{lemma2.6}
For $n\geq2$, and $m\geq2$, there exists some constant $\varepsilon^*>0$, such that, for $0<\varepsilon<\varepsilon^*$, we have
\begin{align}\label{a_11m}
\frac{1}{C\varepsilon}\left(|\Sigma'|+\rho_{n}^{m}(\varepsilon)\right)\leq-a_{ii}\leq \frac{C}{\varepsilon}\left(|\Sigma'|+\rho_{n}^{m}(\varepsilon)\right),\quad i=1,2.
\end{align}
In particular, if $|\Sigma'|>0$, then for $0<\varepsilon<\varepsilon^*_0:=\min\{(\rho_n^m)^{-1}(|\Sigma'|),~\varepsilon^*\}$, we have
\begin{align*}
\frac{|\Sigma'|}{C\varepsilon}\leq-a_{ii}\leq \frac{C|\Sigma'|}{\varepsilon},\quad i=1,2.
\end{align*}
\end{lemma}
\begin{proof}
Similarly as the proof of Lemma \ref{lemma_a11}, we only need to estimate $a_{11}=-\int_{\Omega}|\nabla v_{1}|^{2}$. We mainly deal with the middle term $\int_{\Omega_{R_1}\setminus\Sigma}|\nabla v_{1}|^{2}$, because the first and last term, the estimates for $\int_{\Sigma}|\nabla v_{1}|^{2}$ and $\int_{\Omega\setminus \Omega_{R_1}}|\nabla v_{1}|^{2}$ are the same as \eqref{first_term} and \eqref{R3}. The following constant $C$ is independent on $\varepsilon$, $R_0$ and $|\Sigma'|$. In view of (\ref{v1-bounded1'}), we have
\begin{align*}
&\int_{B'_{R_{1}}\setminus\Sigma'}
\int_{h_2(x')}^{\varepsilon+h_1(x')}\frac{1}{C(\varepsilon+d^m(x'))^{2}}dx_ndx'\\
&\leq
 \int_{\Omega_{R_1}\setminus\Sigma}|\nabla v_{1}|^{2}\\
 &\leq\int_{B'_{R_{1}}\setminus\Sigma'}
\int_{h_2(x')}^{\varepsilon+h_1(x')}\frac{C}{(\varepsilon+d^m(x'))^2}dx_ndx',
\end{align*}
that is,
\begin{align*}
\int_{B'_{R_{1}}\setminus\Sigma'}\frac{dx'}{C(\varepsilon+d^m(x'))}\leq
\int_{\Omega_{R_1}\setminus\Sigma}|\nabla v_{1}|^{2}\leq\int_{B'_{R_{1}}\setminus\Sigma'}\frac{Cdx'}{\varepsilon+d^m(x')}.
\end{align*}

{\bf Case 1.}  $m>n-1$. Using the Young's inequality, we have

\begin{align*}
\int_{B'_{R_{1}}\setminus\Sigma'}\frac{dx'}{\varepsilon+d^m(x')}&\leq\int_{B'_{R_{1}}\setminus B'_{r_{0}}}\frac{dx'}{\varepsilon+dist^{m}(x',B'_{R_{0}})}\\
&\leq\int_{r_{0}}^{R_{0}}\frac{Cr^{n-2}}{\varepsilon}dr+\int_{R_{0}}^{R_{1}}\frac{Cr^{n-2}}{\varepsilon+(r-R_{0})^{m}}dr\\
&\leq\frac{C(R_0^{n-1}-r_0^{n-1})}{\varepsilon}+\int_{0}^{R_1-R_0}\frac{C(t+R_{0})^{n-2}}{\varepsilon+t^m}dt\\
&\leq C\left(\frac{R_0^{n-1}}{\varepsilon}+\int_{0}^{R_1-R_0}\frac{R_{0}^{n-2}}{\varepsilon+t^m}dt
+\int_{0}^{R_1-R_0}\frac{t^{n-2}}{\varepsilon+t^m}dt\right)\\
&\leq C\left(\frac{R_0^{n-1}}{\varepsilon}+\varepsilon^{\frac{1-m}{m}}R_{0}^{n-2}+\varepsilon^{\frac{n-1-m}{m}}\right)\\
&\leq C\left(\frac{|\Sigma'|}{\varepsilon}+\varepsilon^{\frac{n-1-m}{m}}\right).
\end{align*}
On the other hand, similar as before, for a point $p\in\partial\Sigma'$, construct a small cone $Q\subset B'_{R_{1}}\setminus\Sigma'$ with $p$ as the vertex,  such that $dist(x',\Sigma')\leq dist(x',p)$ whenever $x'\in Q$. Then for sufficiently small $\varepsilon$,
\begin{align*}
\int_{B'_{R_{1}}\setminus\Sigma'}\frac{dx'}{\varepsilon+d^m(x')}&\geq\int_{Q}\frac{dx'}{\varepsilon+dist^{m}(x',p)}\\
&\geq\frac{1}{C}\int_{0}^{R_{1}-R_{0}}\frac{r^{n-2}}{\varepsilon+r^{m}}dr\\
&\geq\frac{1}{C}\varepsilon^{\frac{n-1-m}{m}}.
\end{align*}
Thus, we obtain
\begin{align*}
\frac{1}{C}\left(\frac{|\Sigma'|}{\varepsilon}+\varepsilon^{\frac{n-1-m}{m}}\right)\leq-a_{11}\leq C\left(\frac{|\Sigma'|}{\varepsilon}+\varepsilon^{\frac{n-1-m}{m}}\right),
\end{align*}
that is, \eqref{a_11m} for $m>n-1$.

{\bf Case 2.}  $m=n-1$. Similarly, we have for $0<\varepsilon<\frac{1}{e}$,
\begin{align*}
\int_{B'_{R_{1}}\setminus\Sigma'}\frac{dx'}{\varepsilon+d^m(x')}
&\leq C\left(\frac{R_0^{n-1}}{\varepsilon}+\int_{0}^{R_1-R_0}\frac{R_{0}^{n-2}}{\varepsilon+t^{n-1}}dt+\int_{0}^{R_1-R_0}\frac{t^{n-2}}{\varepsilon+t^{n-1}}dt\right)\\
&\leq C\left(\frac{|\Sigma'|}{\varepsilon}+|\ln\varepsilon|\right).
\end{align*}
For the lower bound, similarly as above, for sufficiently small $\varepsilon$, we have
\begin{align*}
\int_{B'_{R_{1}}\setminus\Sigma'}\frac{dx'}{\varepsilon+d^m(x')}&\geq\int_{Q}\frac{dx'}{\varepsilon+dist^{m}(x',p)}\\
&\geq\frac{1}{C}\int_{0}^{R_{1}-R_{0}}\frac{r^{n-2}}{\varepsilon+r^{n-1}}dr\\
&\geq\frac{1}{C}|\ln\varepsilon|.
\end{align*}
Thus,
\begin{align*}
 \frac{1}{C}\left(\frac{|\Sigma'|}{\varepsilon}+|\ln\varepsilon|\right)\leq -a_{11}\leq C\left(\frac{|\Sigma'|}{\varepsilon}+|\ln\varepsilon|\right),
\end{align*}
that is, \eqref{a_11m} for $m=n-1$.

{\bf Case 3.}  $m<n-1$. As above,
\begin{align*}
\int_{B'_{R_{1}}\setminus\Sigma'}\frac{dx'}{\varepsilon+d^m(x')}
&\leq C\left(\frac{R_0^{n-1}}{\varepsilon}+\int_{0}^{R_1-R_0}\frac{R_{0}^{n-2}}{\varepsilon+t^m}dt+\int_{0}^{R_1-R_0}\frac{t^{n-2}}{\varepsilon+t^m}dt\right)\\
&\leq C\left(\frac{|\Sigma'|}{\varepsilon}+1\right),
\end{align*}
and
\begin{align*}
\int_{B'_{R_{1}}\setminus\Sigma'}\frac{dx'}{\varepsilon+d^m(x')}&\geq\int_{Q}\frac{dx'}{\varepsilon+dist^{m}(x',p)}\\
&\geq\frac{1}{C}\int_{0}^{R_{1}-R_{0}}\frac{r^{n-2}}{\varepsilon+r^{m}}dr\\
&\geq\frac{1}{C}.
\end{align*}
Thus, we have that for sufficiently small $\varepsilon$,
\begin{align*}
 \frac{1}{C}\left(\frac{|\Sigma'|}{\varepsilon}+1\right)\leq -a_{11}\leq C\left(\frac{|\Sigma'|}{\varepsilon}+1\right),
\end{align*}
that is, \eqref{a_11m} for $m<n-1$. Lemma \ref{lemma2.6} is proved.
\end{proof}

\section{Proof of Theorem \ref{thm2}}\label{extend}

We decompose the solution $u$ of \eqref{equinfty2} as follows
\begin{equation*}
u(x)=(C_{1}-\varphi(0))v_{1}(x)+v_{0}(x)+\varphi(0),\qquad~x\in\,\widetilde{\Omega},
\end{equation*}
where $v_0\in{C}^{2}(\widetilde{\Omega})$ is the solution of (\ref{equ_v0}) and $v_{1}\in{C}^{2}(\widetilde{\Omega})$ satisfies
\begin{equation*}
\begin{cases}
\Delta{v}_{1}=0,&\mbox{in}~\widetilde{\Omega},\\
v_{1}=1,&\mbox{on}~\partial{D}_{1},\\
v_{1}=0,&\mbox{on}~\partial{D}.
\end{cases}
\end{equation*}
Then we have
\begin{equation}\label{decomposition_u22}
\nabla{u}(x)=(C_{1}-\varphi(0))\nabla{v}_{1}(x)+\nabla{v}_{0}(x),\qquad~x\in\,\widetilde{\Omega}.
\end{equation}
Since $u=C_{1}$ on $\partial{D}_{1}$ and $\|u\|_{H^{1}(\Omega)}\leq\,C$ (independent of $\varepsilon$), by using the trace embedding theorem,
\begin{equation*}
|C_{1}|\leq\,C.
\end{equation*}
We need to estimate $|\nabla v_{1}|$, $|\nabla v_{0}|$ and $|C_{1}-\varphi(0)|$, respectively.

\subsection{Outline of the proof of Theorem \ref{thm2}}

Similarly to the proof in Subsection \ref{proof of thm1}, we have the same estimates of $|\nabla v_{1}|$ as in Proposition \ref{prop1}.
Defining
\begin{equation*}
-a_{11}:=\int_{\widetilde\Omega}|\nabla v_{1}|^{2}dx.
\end{equation*}
So the estimate of $-a_{11}$ is the same as in Lemma \ref{lemma_a11}. Besides, we need the following estimates of $|\nabla v_{0}|$.

\begin{prop}\label{prop2}
Assume the above, let $v_{0}\in{H}^1(\widetilde\Omega)$ be the
weak solution of \eqref{equ_v0}, there exists some constant $\varepsilon_{1}^*>0$, such that, if $0<\varepsilon<\varepsilon_{1}^*$, we have
\begin{align}
|\nabla v_{0}(x)|&\leq\,\frac{C|\varphi(x^{\prime},h(x'))-\varphi(0)|}{\varepsilon+dist^2(x',\Sigma')}+C||\varphi||_{C^{2}(\partial D)},\quad x\in\Omega_{R_{1}},\label{nabla_v0}\\
|\nabla v_{0}(x)|&\leq\,C||\varphi||_{C^{2}(\partial D)},\quad\,x\in\widetilde\Omega\setminus\Omega_{R_{1}}.\label{nabla_v0_22}
\end{align}
\end{prop}
 The proof will be given later. We first use it to prove Theorem \ref{thm2}.

\begin{proof}[\bf Proof of Theorem \ref{thm2}.] Recalling \eqref{decomposition_u22}, by the third line of \eqref{equinfty2}, we have
$$(C_{1}-\varphi(0))\int_{\partial{D}_{1}}\frac{\partial{v}_{1}}{\partial\nu}+\int_{\partial{D}_{1}}\frac{\partial{v}_{0}}{\partial\nu}=0.$$
Recalling the definition of $v_{1}$, we have
\begin{equation*}
\int_{\partial{D}_{1}}\frac{\partial{v}_{1}}{\partial\nu}=\int_{\partial{D}_{1}}\frac{\partial{v}_{1}}{\partial\nu}v_{1}=-\int_{\widetilde\Omega}|\nabla{v}_{1}|^{2}=a_{11}.
\end{equation*}
Hence, 
\begin{align*}|C_{1}-\varphi(0)|=\frac{|\widetilde{Q}[\varphi]|}{|a_{11}|}.
\end{align*}
Thus, 
$$|\nabla{u}|\leq\frac{|\widetilde{Q}[\varphi]|}{|a_{11}|}|\nabla v_{1}|+|\nabla v_{0}|.$$
Using \eqref{v1-bounded1}-\eqref{v1--bounded1} for $|\nabla v_{1}|$, Proposition \ref{prop2} and the estimates for $-a_{11}$, we obtain \eqref{upperbound2'}-\eqref{upperbound2 2'}. Theorem \ref{thm2} is proved.
\end{proof}

\subsection{Proof of Proposition \ref{prop2}}

We introduce a function $\hat{u}\in C^{2}(\mathbb R^{n})$, such that $\hat{u}=0$ on $\partial{D}_{1}$, $\hat{u}=\varphi(x)-\varphi(0\,)$ on $\partial{D}$,
\begin{align*}
\hat{u}(x)
=(1-\bar{u}_{1}(x))\Big(\varphi(x^{\prime},h(x^{\prime}))-\varphi(0)\Big),\quad\mbox{in}~~\Omega_{R_{1}},
\end{align*}
and
\begin{equation*}
\|\hat{u}\|_{C^{2}(\mathbb{R}^{n}\setminus \Omega_{R_{1}})}\leq\,C.
\end{equation*}

By a direct calculation,  for $x\in\Sigma$,
\begin{align*}
&\partial_{x_{i}}\hat{u}=(1-\bar{u}_{1})\partial_{x_{i}}\varphi(x',0),\ \ \quad\quad i=1,2,\cdots,n-1,\\
&\partial_{x_{n}}\hat{u}=-\frac{1}{\varepsilon}\Big(\varphi(x^{\prime},0)-\varphi(0)\Big),\\
&\partial_{x_{i}x_{i}}\hat{u}=(1-\bar{u}_{1})\partial_{x_{i}x_{i}}\varphi(x',0), \quad i=1,2,\cdots,n-1,\\
&\partial_{x_{n}x_{n}}\hat{u}=0.
\end{align*}
Then
\begin{align}
&|\nabla\hat{u}(x)|\leq\frac{|\varphi(x^{\prime},0)-\varphi(0)|}{\varepsilon}+C\|\nabla\varphi\|_{L^\infty(\partial D\cap \Sigma)},\quad x\in\Sigma,\label{nabla_hat_u}
\end{align}
and
\begin{align}
&\Delta \hat{u}(x)=(1-\bar{u}_{1})\Delta_{x'}\varphi(x',0),\quad x\in\Sigma.\label{delta_hat_u}
\end{align}
For  $x\in\Omega_{R_{1}}\setminus\Sigma$, and  $ i=1,2,\cdots,n-1,$
\begin{align*}
\partial_{x_{i}}\hat{u}&=\left(1-\bar{u}_{1}\right)\Big(\partial_{x_{i}}\varphi
+\partial_{x_{n}}\varphi\partial_{x_{i}}h\Big)-\partial_{x_{i}}\bar{u}_{1}~\Big(\varphi(x',h(x'))-\varphi(0))\Big),  \\
\partial_{x_{n}}\hat{u}&=-\partial_{x_{n}}\bar{u}_{1}~\Big(\varphi(x',h(x'))-\varphi(0))\Big),
\end{align*}
and
\begin{align*}
\partial_{x_{i}x_{i}}\hat{u}&=\left(1-\bar{u}_{1}\right)\Big(\partial_{x_{i}x_{i}}\varphi+2\partial_{x_{n}x_{i}}\varphi\partial_{x_{i}}h
+\partial_{x_{n}x_{n}}\varphi(\partial_{x_{i}}h)^2+\partial_{x_n}\varphi \partial_{x_ix_i}h\Big)\\
&\quad-2\partial_{x_{i}}\bar{u}_{1}(\partial_{x_i}\varphi+\partial_{x_n}\varphi\partial_{x_i}h)-\partial_{x_ix_i}\bar{u}_{1}(\varphi(x',h(x'))-\varphi(0)), \\
\partial_{x_{n}x_{n}}\hat{u}&=0.
\end{align*}
Then, for $x\in \Omega_{R_1}\setminus\Sigma$,
\begin{align}
&|\nabla\hat{u}(x)|\leq\frac{C|\varphi(x^{\prime},h(x'))-\varphi(0)|}{\varepsilon+d^2(x')}+C\|\nabla\varphi\|_{L^\infty(\partial D\cap \Omega_{R_1})},\label{nabla_hat_u'}\\
&|\Delta \hat{u}(x)|\leq C\left(\frac{1}{\varepsilon+d^2(x')}+1\right)\|\varphi\|_{C^2(\partial D)}.\label{delta_hat_u'}
\end{align}

Denote
\begin{equation*}
w_{0}:= v_{0}-\hat{u}.
\end{equation*}
Then by the definition of $v_{0}$, \eqref{equ_v0},
\begin{equation}\label{w}
\begin{cases}
-\Delta {w}_{0}=\Delta\hat{u},\quad&\mbox{in}~\widetilde\Omega,\\
 w_{0}=0,\quad\quad&\mbox{on}~\partial\widetilde\Omega.
\end{cases}
\end{equation}
Similarly, as \eqref{nabla_w_out} and \eqref{w_bdd}, we have
\begin{equation*}
\|\nabla {w}_{0}\|_{L^{\infty}(\widetilde\Omega\setminus\Omega_{R_1})}\leq\,C,
\end{equation*}
and
\begin{equation*}
\| w_{0}\|_{L^{\infty}(\tilde\Omega)}\leq\,C||\varphi||_{L^{\infty}(\partial D)}.
\end{equation*}
In order to prove \eqref{nabla_v0}-\eqref{nabla_v0_22}, we only need to prove
\begin{equation*}
\left\|\nabla {w}_{0}\right\|_{L^{\infty}(\Omega_{R_{1}})}\leq\,C.
\end{equation*}
Firstly, multiplying the equation in \eqref{w} by $w_{0}$ and integrating by parts, it follows from \eqref{delta_hat_u} and \eqref{delta_hat_u'} that
\begin{align*}
\int_{\widetilde{\Omega}}|\nabla{w}_{0}|^{2}
=\int_{\widetilde{\Omega}}w_{0}\left(\Delta\hat{u}\right)
\leq\,\|w_{0}\|_{L^{\infty}(\widetilde{\Omega})}\left(\int_{\Omega_{R_{1}}\setminus\Sigma}|\Delta\hat{u}|+C\right)
\leq\,C.
\end{align*}
Instead of \eqref{FsFt11}, we obtain
\begin{align}\label{FsFt22}
\int_{\widehat{\Omega}_{t}(z')}|\nabla{w}_{0}|^{2}\leq\,\frac{C}{(s-t)^{2}}\int_{\widehat{\Omega}_{s}(z')}|w_{0}|^{2}
+(s-t)^{2}\int_{\widehat{\Omega}_{s}(z')}\left|\Delta\hat{u}\right|^{2}.
\end{align}

{\bf Case 1.} For $z'\in\Sigma'_{-\sqrt{\varepsilon}}$  and $0<s<\sqrt{\varepsilon}$. Using the assumption on $\varphi$, we have
\begin{align*}
\int_{\widehat{\Omega}_{s}(z')}|\Delta\hat{u}|^{2}&=\int_{|x'-z'|<s}\int_{0}^{\varepsilon}(1-\bar{u}_{1})^{2}(\Delta_{x'}\varphi(x',0))^{2}dx_{n}dx'\\
&=\int_{|x'-z'|<s}(\Delta_{x'}\varphi(x',0))^{2}dx'\int_{0}^{\varepsilon}(1-\bar{u}_{1})^{2}dx_{n}\\
&=C\varepsilon s^{n-1}||\varphi||_{C^2(\partial D)}^{2},
\end{align*}
and
$$\int_{\widehat{\Omega}_{s}(z')}|w_{0}|^{2}\leq C\varepsilon^{2}\int_{\widehat{\Omega}_{s}(z')}|\nabla w_{0}|^{2}.$$
So \eqref{FsFt22} becomes
\begin{align*}
\int_{\widehat{\Omega}_{t}(z')}|\nabla w_{0}|^{2}
&\leq\left(\frac{C\varepsilon}{s-t}\right)^{2}\int_{\widehat{\Omega}_{s}(z')}|\nabla w_{0}|^{2}+C\varepsilon(s-t)^{2}s^{n-1}||\varphi||_{C^2(\partial D)}^{2}.
\end{align*}
Similarly, as in the steps 1.2--1.3 in the proof of Proposition \ref{prop1}, we obtain
$$\int_{\widehat{\Omega}_{\delta}(z')}|\nabla w_{0}|^{2}\leq C\varepsilon^{n+2}||\varphi||_{C^{2}(\partial D)}^{2},$$
and
$$|\nabla w_{0}(z',z_{n})|\leq C\varepsilon||\varphi||_{C^{2}(\partial D)},\quad\forall~0<z_{n}<\varepsilon.$$

{\bf Case 2.} For $z'\in \Sigma'_{\sqrt{\varepsilon}}\setminus\Sigma'_{-\sqrt{\varepsilon}}$ and $0<s<\sqrt{\varepsilon}$.
In view of (\ref{delta_hat_u'}), we have
\begin{align*}
  \int_{\widehat{\Omega}_{s}(z')}|\Delta\hat{u}|^{2}&\leq C\|\varphi\|_{C^2(\partial D)}^2\left(\int_{|x'-z'|<s}\int_{h(x')}^{\varepsilon+h_1(x')} \frac{1}{(\varepsilon+d^2(x'))^2}dx_ndx'+\varepsilon s^{n-1}\right)\nonumber\\
  &\leq C\|\varphi\|_{C^2(\partial D)}^2\left(\int_{|x'-z'|<s}\frac{1}{\varepsilon+d^2(x')}dx'+\varepsilon s^{n-1}\right)\nonumber\\
  &\leq \frac{Cs^{n-1}}{\varepsilon}\|\varphi\|_{C^2(\partial D)}^2.
\end{align*}
Notice that
$$\int_{\widehat{\Omega}_{s}(z')}|w_{0}|^{2}\leq C\varepsilon^{2}\int_{\widehat{\Omega}_{s}(z')}|\nabla w_{0}|^{2}.$$
By using the similar method as steps 1.2--1.3 in the proof of Proposition 2.1, we obtain that
$$|\nabla w_{0}(z',z_{n})|\leq C||\varphi||_{C^{2}(\partial D)},\quad\forall~h(z')<z_{n}<\varepsilon+h_1(z').$$
{\bf Case 3.} For $z'\in B'_{R_1}\setminus\Sigma'_{\sqrt{\varepsilon}}$ and $0<s<\frac{2}{3}d(z')$.
As above, we have
\begin{align*}
  \int_{\widehat{\Omega}_{s}(z')}|\Delta\hat{u}|^{2}
  &\leq \frac{Cs^{n-1}}{d^2(z')}\|\varphi\|_{C^2(\partial D)}^2,
\end{align*}
and
$$\int_{\widehat{\Omega}_{s}(z')}|w_{0}|^{2}\leq Cd^4(z')\int_{\widehat{\Omega}_{s}(z')}|\nabla w_{0}|^{2}.$$
Similarly as above, we obtain
$$|\nabla w_{0}(z',z_{n})|\leq C||\varphi||_{C^{2}(\partial D)},\quad\forall~h(z')<z_{n}<\varepsilon+h_1(z').$$

Thus, (\ref{nabla_v0}) and (\ref{nabla_v0_22}) follow from (\ref{nabla_hat_u}), (\ref{nabla_hat_u'}) and the estimates of $|\nabla w_0|$.
The proof of Proposition \ref{prop2} is completed.

\section{Proof of Theorem \ref{thm1.6}}\label{sec_thm1.6}

Similarly, we decompose the solution $u(x)$ of \eqref{equinftydiv} as follows
\begin{equation*}
u(x)=C_{1}V_{1}(x)+C_{2}V_{2}(x)+V_{3}(x),\qquad~x\in\,\Omega ,
\end{equation*}
where $V_{j}\in{C}^{2}(\Omega)~(j=1,2,3)$, respectively, satisfying
\begin{equation}\label{equ_V1}
\begin{cases}
\partial_{i}\left(A_{ij}(x)\partial_{j}V_{1}\right)=0,&\mathrm{in}~\Omega,\\
V_{1}=1,&\mathrm{on}~\partial{D}_{1},\\
V_{1}=0,&\mathrm{on}~\partial{D_{2}}\cup\partial{D},
\end{cases}
\end{equation}
\begin{equation}\label{equ_V2}
\begin{cases}
\partial_{i}\left(A_{ij}(x)\partial_{j}V_{2}\right)=0,&\mathrm{in}~\Omega,\\
V_{2}=1,&\mathrm{on}~\partial{D}_{2},\\
V_{2}=0,&\mathrm{on}~\partial{D_{1}}\cup\partial{D},
\end{cases}
\end{equation}
\begin{equation}\label{equ_V3}
\begin{cases}
\partial_{i}\left(A_{ij}(x)\partial_{j}V_{3}\right)=0,&\mathrm{in}~\Omega,\\
V_{3}=0,&\mathrm{on}~\partial{D}_{1}\cup\partial{D_{2}},\\
V_{3}=\varphi,&\mathrm{on}~\partial{D}.
\end{cases}
\end{equation}
Then we have
\begin{equation*}
\nabla{u}(x)=(C_{1}-C_{2})\nabla{V}_{1}(x)+C_{2}\nabla({V}_{1}+{V}_{2})(x)+\nabla{V}_{3}(x),\qquad~x\in\,\Omega.
\end{equation*}
We construct an auxiliary function
$\tilde{u}_{1}\in{C}^{2}(\mathbb{R}^{n})$ to fit this general elliptic equation, such that $\tilde{u}_{1}=1$ on
$\partial{D}_{1}$, $\tilde{u}_{1}=0$ on
$\partial{D}_{2}\cup\partial{D}$, in $\Omega_{R_{1}}$,
\begin{align}\label{utilde}
\tilde{u}_{1}(x)&
=\bar{u}_{1}(x)+\frac{\sum_{i=1}^{n-1}A_{ni}(x)\partial_{x_{i}}(h_{1}-h_{2})(x')}{4A_{nn}(x)}\left(\Big(\frac{2x_{n}-(\varepsilon+h_{1}(x')+h_{2}(x'))}{\varepsilon+h_{1}(x')-h_{2}(x')}\Big)^{2}-1\right),
\end{align}
and
\begin{equation*}
\|\tilde{u}_{1}\|_{C^{2}(\Omega\setminus \Omega_{R_{1}})}\leq\,C.
\end{equation*}
Similarly, we define $\tilde{u}_{2}=1$ on $\partial D_{2}$, $\tilde{u}_{2}=0$ on $\partial{D}_{1}\cup\partial{D}$, $\tilde{u}_{2}=1-\tilde{u}_{1}$ in $\Omega_{R_{1}}$, and $\|\tilde{u}_{2}\|_{C^{2}(\Omega\setminus \Omega_{R_{1}})}\leq\,C$.
Using the assumptions on $h_{1}$ and $h_{2}$, \eqref{h1h2'}--\eqref{h1h3}, a direct calculation still gives
\begin{align*}
\left|\nabla\tilde{u}_{1}(x)\right|&=\frac{1}{\varepsilon},\quad x\in\Sigma,\\
\frac{1}{C\big(\varepsilon+d^{2}(x')\big)}&\leq\left|\nabla\tilde{u}_{1}(x)\right|\leq\frac{C}{\varepsilon+d^{2}(x')},\quad x\in\Omega_{R_{1}}\setminus\Sigma.
\end{align*}
More importantly, thanks to the corrector term in \eqref{utilde}, we obtain the following bound
\begin{align}\label{f_tilde}
\partial_{i}(A_{ij}(x)\partial_{j}\tilde{u}_{1}(x))=0,\quad x\in\Sigma,
\end{align}
and
\begin{align}\label{f_tilde2}
\left|\partial_{i}(A_{ij}(x)\partial_{j}\tilde{u}_{1}(x))\right|\leq\frac{C}{\varepsilon+d^{2}(x')},\quad x\in\Omega_{R_{1}}\setminus\Sigma,
\end{align}
the same as \eqref{nabla2u_bar}. This is important to prove the following Proposition.

\begin{prop}\label{prop1'}
Assume the above, let $V_{1}, V_{2}, V_{3}\in{H}^1(\Omega)$ be the
weak solution of \eqref{equ_V1}, \eqref{equ_V2} and \eqref{equ_V3}, respectively. Then
\begin{equation}\label{nabla_w_i0''}
\|\nabla(V_{i}-\tilde{u}_{i})\|_{L^{\infty}(\Omega)}\leq\,C,~i=1,2,
\end{equation}
consequently,
\begin{equation}\label{V1_bounded1}
\frac{1}{C\varepsilon}\leq|\nabla V_{i}(x)|\leq\,\frac{C}{\varepsilon},\qquad\,x\in\Sigma,~i=1,2,
\end{equation}
\begin{equation}\label{V1-bounded1}
\frac{1}{C\left(\varepsilon+d^{2}(x')\right)}\leq|\nabla V_{i}(x)|\leq\,\frac{C}{\varepsilon+d^{2}(x')},\qquad\,x\in\Omega_{R_{1}}\setminus\Sigma,~i=1,2;
\end{equation}
\begin{equation}\label{V1--bounded1}
|\nabla V_{i}(x)|\leq\,C,\qquad\,x\in\Omega\setminus\Omega_{R_{1}},~i=1,2,
\end{equation}
and
\begin{equation}\label{V1+V2_bounded1}
|\nabla(V_{1}+V_{2})(x)|\leq\,C,\qquad\,x\in\Omega,
\end{equation}
\begin{equation}\label{nabla_V3}
|\nabla{V}_{3}(x)|\leq C||\varphi||_{L^{\infty}(\partial D)},\quad\,x\in\Omega,
\end{equation}
where $C$ is a {\it universal constant}, independent of $|\Sigma'|$.
\end{prop}

\begin{proof}[{\bf Proof of Proposition \ref{prop1'}}]

\noindent\textbf{STEP 1.} Proof of prove \eqref{nabla_w_i0''}. We prove it for $i=1$ and $i=2$ is the same.

Let
$$\widetilde{w}_{1}=V_{1}-\tilde{u}_{1}.$$
Similarly, instead of \eqref{w20}, we have
\begin{equation}\label{divergence w1}
\begin{cases}
-\displaystyle\partial_{i}(A_{ij}(x)\partial_{j} \widetilde{w}_{1})=\displaystyle\partial_{i}(A_{ij}(x)\partial_{j} \tilde{u}_{1})=:\tilde{f},\quad&\mbox{in }\Omega,\\
\widetilde{w}_{1}=0, &\mbox{on }\partial \Omega.
\end{cases}
\end{equation}
By the standard elliptic theory,
\begin{equation*}
|\widetilde{w}_{1}|+\left|\nabla\widetilde{w}_{1}\right|\leq\,C,
\quad\mbox{in}~~ \Omega\setminus\Omega_{R_{1}}.
\end{equation*}
On the other hand, by the maximum principle, we have
\begin{equation}\label{wtilde_bdd}
\|\widetilde{w}_{1}\|_{L^{\infty}(\Omega)}\leq\,C.
\end{equation}

\textbf{STEP 1.1. Boundedness of the energy.}
Multiplying the equation in (\ref{divergence w1}) by $\tilde{w}_{1}$, integrating by parts, using (\ref{ene01a}), \eqref{f_tilde}, \eqref{f_tilde2} and \eqref{wtilde_bdd}, we have
\begin{align*}
\lambda\int_{\Omega}|\nabla \widetilde{w}_{1}|^{2}dx
&\leq\int_{\Omega}A_{ij}\partial_{i}\widetilde{w}_{1}\partial_{j}\widetilde{w}_{1}dx\\
&=\,\int_{\Omega}\tilde{f}\,\widetilde{w}_{1}dx\\
&\leq\,\|\widetilde{w}_{1}\|_{L^{\infty}(\Omega)}\left(\int_{\Omega_{R_{1}}\setminus\Sigma}|\tilde{f}|dx+C\right)\leq\,C.
\end{align*}
So that
\begin{equation*}
\int_{\Omega}\left|\nabla \widetilde{w}_{1}\right|^{2}dx\leq\,C.
\end{equation*}

\textbf{STEP 1.2. Local energy estimates.}
Multiplying the equation in (\ref{divergence w1}) by $\eta^2\widetilde{w}_{1}$, where $\eta$ is the same cut-off function defined in the step 1.2 of proof of Proposition \ref{prop1}, and integrating by parts, we deduce
\[\int_{\widehat{\Omega}_{s}(z^{\prime})}\displaystyle A_{ij}\partial_{j} \widetilde{w}_{1} \partial_{i}(\eta^2\widetilde{w}_{1})dx=\int_{\widehat{\Omega}_{s}(z^{\prime})}\displaystyle \tilde{f}\eta^2\widetilde{w}_{1}dx.\]
Then
\begin{equation*}
\begin{aligned}
\int_{\widehat{\Omega}_{s}(z^{\prime})}\displaystyle A_{ij}(\eta\partial_{j}\widetilde{w}_{1})(\eta \partial_{i}\widetilde{w}_{1})dx
&=\int_{\widehat{\Omega}_{s}(z^{\prime})}\bigg[-2\displaystyle A_{ij}(\eta\partial_{j}\widetilde{w}_{1})\widetilde{w}_{1}\partial_{i}\eta+\tilde{f}\eta^{2}\widetilde{w}_{1}\bigg]dx.
\end{aligned}
\end{equation*}
By \eqref{ene01a} and the Cauchy inequality,
\begin{align*}
\lambda\int_{\widehat{\Omega}_{s}(z^{\prime})}|\eta\nabla\widetilde{w}_{1}|^2dx&\leq
\frac{\lambda}{4}\int_{\widehat{\Omega}_{s}(z^{\prime})}|\eta\nabla\widetilde{w}_{1}|^2dx\\
&\quad+\frac{C}{(s-t)^{2}}\int_{\widehat{\Omega}_{s}(z^{\prime})}|\widetilde{w}_{1}|^{2}dx+C(s-t)^{2}\int_{\widehat{\Omega}_{s}(z^{\prime})}|\tilde{f}|^{2}dx.
\end{align*}
Thus,
\begin{equation*}
\int_{\widehat{\Omega}_{t}(z^{\prime})}|\nabla \widetilde{w}_{1}|^2dx\leq\frac{C}{(s-t)^2}\int_{\widehat{\Omega}_{s}(z^{\prime})}|\widetilde{w}_{1}|^2dx+C(s-t)^{2}\int_{\widehat{\Omega}_{s}(z^{\prime})}|\tilde{f}|^2dx.
\end{equation*}
Using the iteration argument, similarly as step 1.2 in the proof of Proposition \ref{prop1}, we have $\widetilde{w}_{1}$ also satisfies \eqref{energy_w_inomega_z1}, that is,
\begin{equation*}
\int_{\widehat{\Omega}_{\delta}(z')}\left|\nabla{\widetilde{w}_{1}}\right|^{2}dx\leq
C\delta^{n},
\end{equation*}
where $\delta=\delta(z')$.
Thus, similarly as step 1.3 in the proof of Proposition \ref{prop1},  \eqref{nabla_w_i0''} is established.

\noindent\textbf{STEP 2.} Proof of \eqref{V1+V2_bounded1} and \eqref{nabla_V3} are the same as step 2 in the proof of Proposition \ref{prop1}. Proposition \ref{prop1'} is established.
\end{proof}

\begin{proof}[\bf Proof of Theorem \ref{thm1.6}]

Define
$$a_{kl}:=\int_{\partial D_{k}}A_{ij}(x)\,\partial_{j}{V}_{l}\,\nu_{i},\quad k,l=1,2.$$
By integrating by parts,
\begin{align*}
0=\int_{\Omega}\partial_{i}(A_{ij}(x)\partial_{j}V_{1})\cdot\,V_{1}
=&-\int_{\Omega}A_{ij}(x)\partial_{j}V_{1}\partial_{i}V_{1}-\int_{\partial{D}_{1}}A_{ij}(x)\,\partial_{j}{V}_{1}\,\nu_{i}\cdot\,1\\
=&-\int_{\Omega}A_{ij}(x)\partial_{i}V_{1}\partial_{j}V_{1}-a_{11}.
\end{align*}
That is,
$$-a_{11}=\int_{\Omega}A_{ij}(x)\partial_{i}V_{1}\partial_{j}V_{1}.$$
By the uniform ellipticity condition \eqref{ene01a},
$$\lambda\int_{\Omega}|\nabla\,V_{1}|^{2}\leq\,-a_{11}\leq\Lambda\int_{\Omega}|\nabla\,V_{1}|^{2}.$$
Thus Lemma \ref{lemma_a11} holds still. Then, combining with Proposition \ref{prop1'}, the proof of Theorem \ref{thm1.6} is completed.
\end{proof}

\noindent{\bf{\large Acknowledgements.}} Part of this work was completed while the third author was visiting Professor Hongjie Dong at Brown University. She also would like to thank the Division of Applied Mathematics at Brown University for the hospitality and the stimulating environment. The authors would like to express their gratitude to Professor Hongjie Dong, Theorem \ref{thm1.6} is added thanks to his comments.



\begin{thebibliography}{99}



\bibitem{akl} H. Ammari; H. Kang; M. Lim, Gradient estimates to the conductivity problem. Math.
Ann. 332 (2005), 277-286.

\bibitem{ackly} H. Ammari; G. Ciraolo; H. Kang; H. Lee; K. Yun, Spectral analysis of the Neumann-Poincar\'{e} operator and characterization of the stress concentration in anti-plane elasticity. Arch. Ration. Mech. Anal.  208  (2013),  275-304.


\bibitem{adkl}  H. Ammari; H. Dassios; H. Kang; M. Lim, Estimates for the electric field in the
presence of adjacent perfectly conducting spheres. Quat. Appl. Math. 65
(2007), 339-355.

\bibitem{aklll}  H. Ammari; H. Kang; H. Lee; J. Lee; M. Lim, Optimal estimates for the electrical
field in two dimensions. J. Math. Pures Appl. 88 (2007), 307-324.


\bibitem{akllz} H. Ammari; H. Kang; H. Lee; M. Lim; H. Zribi, Decomposition theorems and fine estimates for electrical
fields in the presence of closely located circular inclusions. J.
Differential Equations 247 (2009), 2897-2912.


\bibitem{basl} I. Babu\u{s}ka; B. Andersson; P. Smith; K.
Levin, Damage analysis of fiber composites. I. Statistical analysis
on fiber scale. Comput. Methods Appl. Mech. Engrg. 172 (1999), 27-77.


\bibitem{bly1} E. Bao; Y.Y. Li; B. Yin, Gradient estimates for the perfect conductivity problem. Arch. Ration. Mech. Anal. 193 (2009), 195-226.

\bibitem{bly2} E. Bao; Y.Y. Li; B. Yin, Gradient estimates for the perfect and insulated conductivity problems with multiple inclusions. Comm. Partial Differential Equations 35 (2010), 1982-2006.

\bibitem{bjl} J.G. Bao; H.J. Ju; H.G. Li, Optimal boundary gradient estimates for Lam\'{e} systems with partially infinite coefficients. Adv. Math. 314 (2017), 583-629.

\bibitem{bll} J.G. Bao; H.G. Li; Y.Y. Li, Gradient estimates for solutions of the Lam\'{e} system with partially infinite coefficients. Arch. Ration. Mech. Anal.  215  (2015),  no. 1, 307-351.

\bibitem{bll2} J.G. Bao; H.G. Li; Y.Y. Li, Gradient estimates for solutions of the Lam\'{e} system with partially infinite coefficients  in dimensions greater than two. Adv. Math. 305 (2017), 298-338.

\bibitem{bt1} E. Bonnetier; F. Triki, Pointwise bounds on the gradient and the spectrum of the Neumann-Poincar\'{e} operator: the case of 2 discs, Multi-scale and high-contrast PDE: from modeling, to mathematical analysis, to inversion, Contemp. Math., 577, Amer. Math. Soc., Providence, RI, 2012, pp. 81-91.

\bibitem{bt2}E. Bonnetier; F. Triki, On the spectrum of the Poincar\'{e} variational problem for two close-to-touching inclusions in 2D. Arch. Ration. Mech. Anal. 209 (2013), no. 2, 541-567.

\bibitem{bv} E. Bonnetier; M. Vogelius, An elliptic regularity result for a composite medium
with ``touching'' fibers of circular cross-section. SIAM J. Math. Anal. 31 (2000), 651-677.


\bibitem{bc} B. Budiansky; G.F. Carrier, High shear stresses in stiff fiber composites, J. App. Mech. 51 (1984), 733-735.

\bibitem{dongli} H.J. Dong; H.G. Li, Optimal estimates for the conductivity problem by Green's function method. arXiv: 1606.02793v1. (2016)


\bibitem{dongzhang} H.J. Dong; H. Zhang, On an elliptic equation arising from composite materials. Arch. Ration. Mech. Anal. 222 (2016), no. 1, 47-89.


\bibitem{gn} Y. Gorb; A. Novikov, Blow-up of solutions to a p-Laplace equation, Multiscal Model. Simul. 10 (2012), 727-743.

\bibitem{gu} Gutiérrez, Cristian E. The Monge-Amp{\`e}re equation. Progress in Nonlinear Differential Equations and their Applications, 44. Birkhäuser Boston, Inc., Boston, MA, 2001.

\bibitem{kly0} H. Kang; H. Lee; K. Yun, Optimal estimates and asymptotics for the stress concentration between closely located stiff inclusions, Math. Ann. 363 (2015), 1281-1306.

\bibitem{kly}  H. Kang; M. Lim; K. Yun, Asymptotics and computation of the solution to the conductivity equation in the presence of adjacent inclusions with extreme conductivities. J. Math. Pures Appl. (9) 99 (2013), 234-249.

\bibitem{kly2}  H. Kang; M. Lim; K. Yun,  Characterization of the electric field concentration between two adjacent spherical perfect conductors.  SIAM J. Appl. Math. 74 (2014), 125-146.

\bibitem{ky} H. Kang; S. Yu, Quantitative characterization of stress concentration in the presence of closely spaced hard inclusions in two-dimensional linear elasticity. arXiv: 1707.02207v2. (2017)

\bibitem{k1} J.B. Keller, Conductivity of a medium containing a dense array of perfectly conducting spheres or cylinders or nonconducting cylinders, J. Appl. Phys., 34 (1963), pp. 991-993.


\bibitem{llby} H.G. Li; Y.Y. Li; E.S. Bao; B. Yin, Derivative estimates of solutions of elliptic systems in narrow regions. Quart. Appl. Math.  72  (2014),  no. 3, 589-596.

\bibitem{LX} H.G. Li and L.J. Xu, Optimal estimates for the perfect conductivity problem with inclusions close to the boundary.  SIAM J. Math. Anal. 49 (2017), no. 4, 3125-3142.



\bibitem{ln}  Y.Y. Li; L. Nirenberg, Estimates for elliptic system from composite material. Comm. Pure Appl. Math. 56 (2003), 892-925.


\bibitem{lv} Y.Y. Li; M. Vogelius, Gradient stimates for solutions to divergence form elliptic equations with discontinuous
coefficients. Arch. Rational Mech. Anal. 153 (2000), 91-151.

\bibitem{lyu} M. Lim; S. Yu, Stress concentration for two nearly touching circular holes. arXiv: 1705.10400v1. (2017)

\bibitem{ly}M. Lim; K. Yun, Strong influence of a small fiber on shear stress in fiber-reinforced composites. J.
Differential Equations 250 (2011), 2402-2439.

\bibitem{ly2} M. Lim; K. Yun, Blow-up of electric fields between closely spaced spherical perfect conductors, Comm. Partial Differential Equations, 34 (2009), pp. 1287-1315.

\bibitem{m} X. Markenscoff, Stress amplification in vanishingly small geometries. Computational Mechanics 19 (1996), 77-83.


\bibitem{y1} K. Yun, Estimates for electric fields blown up between closely adjacent conductors with arbitrary shape. SIAM J. Appl. Math. 67 (2007),  714-730.

\bibitem{y2} K. Yun, Optimal bound on high stresses occurring between stiff fibers with arbitrary shaped cross-sections. J. Math. Anal. Appl. 350 (2009), 306-312.

\bibitem{y3} H. Yun, An optimal estimate for electric fields on the shortest line segment between two spherical insulators in three dimensions. J.
Differential Equations (2016), 261(1): 148-188.

\end{thebibliography}
\end{document}